\documentclass[12pt]{amsart}
\usepackage{amsmath}
\usepackage{amsfonts}
\usepackage{amssymb}
\usepackage[mathscr]{eucal}
\usepackage{graphicx}
\makeatletter
\def\numberwithin#1#2{\@ifundefined{c@#1}{\@nocnterrr}{%
  \@ifundefined{c@#2}{\@nocnterr}{%
  \@addtoreset{#1}{#2}%
  \toks@\expandafter\expandafter\expandafter{\csname the#1\endcsname}%
  \expandafter\xdef\csname the#1\endcsname
    {\expandafter\noexpand\csname the#2\endcsname
     .\the\toks@}}}}
\makeatother
\numberwithin{equation}{section}

\newtheorem{theorem}{Theorem}
\numberwithin{theorem}{section}

\newtheorem{example}[theorem]{Example}
\newtheorem{lemma}[theorem]{Lemma}

\newtheorem{remark}[theorem]{Remark}

\newenvironment{rmk}{\begin{remark} \em}{\end{remark}}

\usepackage{amsmath,amssymb}
\usepackage{graphicx}

\newdimen\tableauside\tableauside=1.0ex
\newdimen\tableaurule\tableaurule=0.4pt
\newdimen\tableaustep
\def\phantomhrule#1{\hbox{\vbox to0pt{\hrule height\tableaurule
width#1\vss}}}
\def\phantomvrule#1{\vbox{\hbox to0pt{\vrule width\tableaurule
height#1\hss}}}
\def\sqr{\vbox{%
  \phantomhrule\tableaustep

\hbox{\phantomvrule\tableaustep\kern\tableaustep\phantomvrule\tableaustep}%
  \hbox{\vbox{\phantomhrule\tableauside}\kern-\tableaurule}}}
\def\squares#1{\hbox{\count0=#1\noindent\loop\sqr
  \advance\count0 by-1 \ifnum\count0>0\repeat}}
\def\tableau#1{\vcenter{\offinterlineskip
  \tableaustep=\tableauside\advance\tableaustep by-\tableaurule
  \kern\normallineskip\hbox
    {\kern\normallineskip\vbox
      {\gettableau#1 0 }%
     \kern\normallineskip\kern\tableaurule}%
  \kern\normallineskip\kern\tableaurule}}
\def\gettableau#1 {\ifnum#1=0\let\next=\null\else
  \squares{#1}\let\next=\gettableau\fi\next}

\newcommand{\be}{\begin{equation}}
\newcommand{\ee}{\end{equation}}
\newcommand{\bea}{\begin{eqnarray}}
\newcommand{\eea}{\end{eqnarray}}
\newcommand{\ba}{\begin{array}}
\newcommand{\ea}{\end{array}}

\newcommand{\id}{\hbox{1\kern-.27em l}}

\newcommand{\al}{\alpha}

\newcommand{\bet}{\beta}

\newcommand{\ka}{\kappa}

\newcommand{\la}{\lambda}

\newcommand{\cN}{\mathcal{N}}

\newcommand{\rar}{\rightarrow}

\newcommand{\non}{\nonumber}

\newcommand{\so}{\mathrm{so}}
\newcommand{\spl}{\mathrm{sp}}

\newcommand{\Spin}{\mathrm{Spin}}

\newtheorem{Pro}{Proposition}

\newtheorem{Def}{Definition}

\begin{document}
\title{Symbol Invariant of   Partition and the Construction}

\author{Bao Shou}
\address{Center  of Mathematical  Sciences, Zhejiang University,
Hangzhou 310027, CHINA}
\email{bsoul@zju.edu.cn}

\thanks{We would like to thank Ming Huang and Xiaobo Zhuang   for  many helpful discussions.  This work was supported by a grant from  the Postdoctoral Foundation of Zhejiang Province.}

\date{ \today}

\begin{abstract}
The symbol is used to describe the Springer correspondence for the classical groups.  We propose  equivalent definitions of symbols for rigid partitions in the $B_n$, $C_n$, and $D_n$ theories uniformly.  Analysing the new definition of symbol in detail,  we give  rules to construct symbol of a partition,   which are easy to remember and to operate on.  We introduce  formal operations of a partition, which   reduce the difficulties in the  proof  of the construction  rules. According these  rules, we give a closed formula of symbols for different  theories uniformly. As applications,  previous  results can be    illustrated  more clearly by  the construction rules of symbol.

\smallskip


\smallskip
\noindent \textbf{MSC (2010)}.05E10
\end{abstract}

\maketitle

%%%%%%%%%%%%%%%%%%%%%%%%%%%%%%%%%%%%%%%%%%%%%%%%%%%%%%%%%%%

\tableofcontents

\section{Introduction}
The Springer correspondence \cite{CM93} is an injective map from the unipotent conjugacy classes of a simple group to the set of unitary representations of the Weyl group. For the classical groups this map can be described explicitly in terms of certain   \textit{symbols} introduced in \cite{CM93}. For the Weyl group in the $B_n$, $C_n$, and $D_n$ theories, irreducible unitary representations are in one-to one correspondence with ordered pairs of partitions $[\al;\bet]$ where $\al$ is a partition of $n_\al$ and $\bet$ is a partition of $n_{\beta}$ satisfying    $n_\al+n_\bet= n$.

Gukov and Witten initiated a study of the $S$ duality of  surface operators in $\cN=4$ super Yang-Mills theories in \cite{GW06}\cite{Wit07}\cite{GW08}\cite{GW14}. The $S$-duality conjecture suggest that surface operators in  the theory with gauge group $G$ should have a counterpart  in the dual theory with gauge group $G^{L}$(the Langlands dual group of $G$)\cite{GW08}\cite{GNO76}. The surface operator is defined by prescribing  a certain singularity structure of fields near the surface supported by the operator.

A subclass of surface operators called {\it 'rigid'} surface operators corresponding to rigid partition  are expected to be closed under $S$-duality. There are two types rigid surface operators:  unipotent and semisimple.  They  correspond to certain ( unipotent and semisimple) conjugacy classes of the (complexified) gauge group. The unipotent  surface operators  are classified by partitions.  The  semisimple surface operators in the theories   are labelled by pair of  partitions.  Unipotent  surface operators arise  when one of these two partitions is empty.  The  symbol is an invariant of partitions under the $S$ duality map. Using the symbol invariant,  Wyllard made some explicit proposals for how the $S$-duality map should act on rigid surface operators in \cite{Wy09}. Other aspects of the surface operators have been studied in \cite{GM07}\cite{DGM08}\cite{GW14}\cite{Sh06}.

In \cite{ShO06}, relying  on the characteristic of symbol, we found that a new  subclass of rigid surface operators in $B_n$ and $C_n$ theories are   related by $S$-duality map.  To find more $S$-duality maps  for rigid surface operators, we need characterize the structure of symbol further.  However, litter work  has been done on it.  The calculation of symbol is a complicated  and tedious work. In \cite{ShO06}, we  simplify the construction of symbol by mapping  a partition into two    partitions  with only  even rows.  A simple and even explicit  formula of  symbol need to be found.

In this paper, we attempt to extend the analysis in \cite{Wy09}\cite{ShO06}  to study the construction of symbol further. Because of no noncentral rigid conjugacy classes in $A_n$ series, we do not discuss surface operators in this case. We also omit exceptional groups, which are more complicated. We will concentrate on the $B_n, C_n$, and $D_n$ series. Our main result  is the following table of computational  rules of symbols. We determine  the contributions to symbol for each row of a partition in the  $B_n(t=-1)$, $C_n(t=0)$, and $D_n(t=1)$ theories  uniformly as follows.
\begin{center}
\begin{tabular}{|c|c|c|c|}\hline
\multicolumn{4}{|c|}{ Contribution to  symbol of the $i$\,th row of a partition }\\ \hline
Parity of  the length of  row & Parity of $i$ &  $L$ &  Contribution to symbol \\ \hline
odd & even  & $\frac{1}{2}(\sum^{m}_{k=i}n_k+1)$ & $\Bigg(\!\!\!\ba{c}0 \;\; 0\cdots \overbrace{ 1\;\; 1\cdots1}^{L} \\
\;\;\;0\cdots 0\;\; 0\cdots 0 \ \ea \Bigg)$   \\ \hline
even & odd    & $\frac{1}{2}(\sum^{m}_{k=i}n_k)$ & $\Bigg(\!\!\!\ba{c}0 \;\; 0\cdots \overbrace{ 1\;\; 1\cdots1}^{L} \\
\;\;\;0\cdots 0\;\; 0\cdots 0 \ \ea \Bigg)$  \\ \hline
even & even    & $\frac{1}{2}(\sum^{m}_{k=i}n_k)$ &  $\Bigg(\!\!\!\ba{c}0 \;\; 0\cdots 0\;\; 0 \cdots 0 \\
\;\;\;0\cdots \underbrace{1 \;\;1\cdots 1}_{L} \ \ea \Bigg)$  \\ \hline
odd & odd     & $\frac{1}{2}(\sum^{m}_{k=i}n_k-1)$  &  $\Bigg(\!\!\!\ba{c}0 \;\; 0\cdots 0\;\; 0 \cdots 0 \\
\;\;\;0\cdots \underbrace{1\; \;1\cdots 1}_{L} \ \ea \Bigg)$  \\ \hline
\end{tabular}
\end{center}
These rules are easy to remember and  operator on, which is equivalent to  the following  closed formula of  symbol.
\begin{Pro}\label{Fbcdi}
For a partition $\lambda=m^{n_m}{(m-1)}^{n_{m-1}}\cdots{1}^{n_1}$, we introduce two notations
$$\Delta^{T}_i=\frac{1}{2}(\sum^{m}_{k=i}n_k+\frac{1+(-1)^{i+1}}{2}),\quad\quad  P^{T}_i=\frac{1+\pi_i}{2}$$
where the superscript $T$ indicates it is related to the top row of the symbol and
$$\pi_i=(-1)^{\sum^{m}_{k=i}n_k}\cdot(-1)^{i+1+t},$$
for $B_n(t=-1)$, $C_n(t=0)$, and $D_n(t=1)$ theories.
Other parallel notations
$$\Delta^{B}_i=\frac{1}{2}(\sum^{m}_{k=i}n_k+\frac{1+(-1)^{i}}{2}),\quad \quad  P^{B}_i=\frac{1-\pi_i}{2}$$
where the superscript $B$ indicates it is  related to the bottom row of the symbol and $P^{B}_i$ is a projection operator similar to  $P^{T}_i$.
Then the symbol $\sigma(\lambda)$   is
\begin{equation}\label{FbFbi}
\sigma(\lambda)=\sum^m_{i=1}\big\{ \Bigg(\!\!\!\ba{c}0\;\;0\cdots 0\;\; \overbrace{1\cdots 1}^{P^{T}_i \Delta^{T}_i}  \\
\;\;\;\underbrace{0\cdots 0 \;\;0 \cdots 0}_{l+t}\ \ea \Bigg)
+
 \Bigg(\!\!\!\ba{c}\overbrace{0\;\;0\cdots 0\;\; 0\cdots 0}^{l}  \\
\;\;\;0\cdots 0 \;\;\underbrace{1\cdots 1}_{P^{B}_i \Delta^{B}_i}\ \ea \Bigg)\big\}
\end{equation}
with  $l=(m+(1-(-1)^m)/2)/2$.
\end{Pro}
The equivalent definition of symbol (Definitions \ref{Dn}) proposed by us paly a critical role. The proof of the computational  rules  is  based on three lemmas (Lemmas \ref{Lsy}, \ref{Lsysy}, and \ref{Lo}).  The three lemmas are clearly proved by  introducing formal operators of partition (Lemmas \ref{Lem} and \ref{Lem-v}).

The following is an outline of  this article. In Section 2, we propose an  equivalent definition of symbols for the $B_n, C_n$,  and $D_n$ theories uniformly. Then we  introduce formal operations of a partition.  These formal operations  are  helpful for  understanding the new definitions of symbol.  In Section 3,  we discuss   the contribution to symbol  of each row in a partition formally.  For an even row, it is natural to calculate the contribution. For an odd row, the calculation  is artificial.  However, it is reasonable  to compute the contributions to symbol of a pair of   odd  rows together.    The new  definition of symbol and the formal operations of a  partition are crucial for  computing the contribution  of odd rows.  We also give  rules to construct symbol which are easy to memory and to operate on. Finally, according these  construction  rules of symbol, we give a closed formula of symbols for the $B_n, C_n$, and $D_n$ theories uniformly. In Section 4, as applications,   we prove several propositions which are consistent with the results   in \cite{ShO06}. The construction rules of symbol are helpful in  searching $S$-duality map of surface operators in $\cN=4$ super-Yang-Mills theories.

In fact, we give another completely new construction in \cite{SW17}.   There is another invariant of partition called fingerprint related to the Kazhdan-Lusztig map\cite{Lu79}\cite{Lu84}\cite{Sp92}. It is assumed that the  fingerprint invariant  is equivalent to the symbol invariant. Hopefully our constructions will be helpful in the proof of the equivalence. The mismatch of the total number of rigid surface operators was found in the $B_n/C_n$ theories through the study of the generating functions \cite{Wy09}\cite{HW07a}\cite{HW07b}\cite{HW08}. Clearly more work is required.

\section{Symbol  of rigid partitions in the $B_n$, $C_n$,  and $D_n$ theories}

Firstly,  we  recall  the definitions of symbols in \cite{CM93}.  The differences of the definitions of symbols between  different theories are obstructions to study the constructions of symbols. And then  find  equivalent  definitions of symbols which are convenient to compute symbols for different theories. Finally,   We  introduce several   formal operations acting on the   partition,  which  are not only helpful to understand the new definitions of symbol but also helpful for the proofs of lemmas in the  next section.

\subsection{Old definition of  symbol}
 A partition $\la$ of the positive integer $n$ is a decomposition $\sum_{i=1}^l \la_i = n$  ($\la_1\ge \la_2 \ge \cdots \ge \la_l$).  The integer $l$  is called the length of the partition. There is a one-to-one correspondence between partition and Young tableaux. For instance the partition $3^22^31$  corresponds to
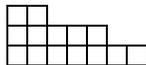
\be
\tableau{2 5 7}\non
\ee
Young diagrams turn out to be extremely useful in the study of symmetric functions and group representation theory. They  are also useful tools for the construction of  the eigenstates of Hamiltonian System \cite{Sh11}.  For two partitions $\la$ and $\ka$,   $\la+\ka$ is the partition with parts $\la_i+\ka_i$.

Unipotent conjugacy classes  are classified by partitions.  They  are in one-to-one correspondence with partitions of $2n{+}1$($2n$) in the $B_n$($D_n$) theories with all even integers appear an even number of times.    And hey  are in one-to-one correspondence with partitions of $2n$   in the $C_n$ theory with  the  odd integers appear an even number of times. A partition in the $B_n$ and $D_n$($C_n$) theories is called {\it rigid} if it has no gaps (i.e.~$\la_i-\la_{i+1}\leq1$ for all $i$) and no odd (even) integer appears exactly twice. We  focus on rigid partitions  in this paper.

A partition $\la$ is called {\it special} if its transpose partitions  satisfy  the following conditions
\begin{eqnarray}  \label{special}
B_n: &\quad \la^t &\;\mbox{is orthogonal} \,,\non \\
C_n: &\quad \la^t &\;\mbox{is symplectic} \,, \\
D_n: &\quad \la^t &\;\mbox{is symplectic} \,.\non
\end{eqnarray}
It is easy to find that all of the  rows of a rigid special partition  in the $B_n$ theory are  odd. And those rows in  the $C_n$ and $D_n$  are  even.

The construction of \textit{symbol} in \cite{CM93} \cite{Wy09}  are introduced as follows
\begin{Def}
\underline{For a partition in the $B_n$ theory: }firstly,  we add $l-k$  to the $k$th part of the partition.
Next we arrange the odd parts and even  parts in  increasing  sequences $2f_i+1$ and  $2g_i$ from \textbf{right to left}, respectively.
Then we calculate terms
 \begin{equation}\label{ab}
   \al_i = f_i-i+1\quad,\quad\quad  \bet_i = g_i-i+1.
 \end{equation}
Finally the {\it symbol} is written  as follows
\be \label{symbol}
\left(\ba{@{}c@{}c@{}c@{}c@{}c@{}c@{}c@{}} \al_1 &&\al_2&&\al_3&& \cdots \\ &\bet_1 && \bet_2 && \cdots  & \ea \right).
\ee

  \textbf{Example:}  $B_{10}$,  $\la=3^3\,2^4\,1^4$, the symbol is
  \be \label{exs}
\left(\ba{@{}c@{}c@{}c@{}c@{}c@{}c@{}c@{}c@{}c@{}c@{}c@{}} 0&&0&&1&&1&&1&&1 \\ & 1 && 1 && 1 &&1&&2& \ea \right).
\ee
The  rows of this symbol can be written  as two partitions $[1^4,21^4]$  corresponding to a unitary representation of the Weyl group.

 \underline{ For a partition in the $C_n$ theory: } we  append an extra $0$ as the last part of the partition if the length of the partition is even. Next we calculate terms   $f_i$ and   $g_i$  as in the $B_n$ case. Then   we calculate   terms $\al_i = g_i-i+1$ and $\bet_i = f_i-i+1$.

  \textbf{Example:}  $C_{10}$,   $\la=3^2\,2^6\,1^2$, the length of the partition is even and  the symbol is
%\footnote{It is the  same as (\ref{exs}), which  will be explained in the next section.}
\be
\left(\ba{@{}c@{}c@{}c@{}c@{}c@{}c@{}c@{}c@{}c@{}c@{}c@{}} 0&&0&&1&&1&&1&&1 \\ & 1 && 1 && 1 &&1&&2& \ea \right).\non
\ee

  \textbf{Example:}   $C_{12}$,  $\la=3^2\,2\,1^4$, the length of the partition is odd and  the symbol is
\be
\left(\ba{@{}c@{}c@{}c@{}c@{}c@{}c@{}c@{}c@{}c@{}c@{}c@{}} &&&&1&&1&&1&&1 \\ &  &&  && 0 &&0&&2& \ea \right)\non
\ee

\underline{For a partition in the $D_n$ theory:} we calculate $f_i$ and $g_i$ exactly as in the $B_n$ case, and  then   calculate   terms $\al_i = g_i-i+1$ and  $\bet_i = f_i-i+1$.  A rigid partitions  in the  $D_n$ theory  always have at least one part equal to 1, so $\bet_1=0$. We  omit this entry and relabel $\beta_2 \rightarrow \beta_1$ etc.

  \textbf{Example:} $D_{10}$,  $\la=4^2\,3\,2^2\,1^5$, the length of the partition is even and  the symbol is
\be
\left(\ba{@{}c@{}c@{}c@{}c@{}c@{}c@{}c@{}c@{}c@{}c@{}c@{}} &&1&&1&&2&&2&&2 \\ &  && 0&& 0 &&0&&2& \ea \right)\non
\ee

\end{Def}

 \begin{rmk}
  \begin{enumerate}
    \item The locations between the entries  $\alpha_*$    and     $\beta_*$     the symbol are not essential.
    \item If not  specified, $l$ is the length of the partition  including the extra  0 appended. The following  table  illustrates the first step of the computation of  symbol
\begin{center}
\begin{tabular}{|r||c|c|c|c|}\hline
        $\lambda_{k}:$           & $\lambda_1$ & $\lambda_{2}$ &$\cdots$ & $\lambda_l$     \\ \hline
 $ l-k: $         & $l-1$ &$ l-2 $&$\cdots$ &$ 0 $                                         \\   \hline
 $l-k+\lambda_{k}: $& $l-1+\lambda_1$ &$ l-2+\lambda_2$&$\cdots $& $\lambda_l$              \\  \hline
\end{tabular}
\end{center}
 Note that $l-k$ is an increasing sequence from right to left,  whose first term is zero. This fact will be used frequently. The  sequences   $f_i$  and   $g_i$   are calculated as follows
 \begin{center}
\begin{tabular}{|r||c|c|c|c|}\hline
  $2f_i+1: $      & $\cdots$ & $2f_2+1$ & $2f_1+1$   \\ \hline
    $ f_i:  $     & $\cdots $ & $f_2 $   & $f_1 $    \\ \hline
\end{tabular}
\qquad
\begin{tabular}{|r||c|c|c|}\hline
 $ 2g_i:$     & $\cdots $ & $2g_2$  &  $2g_1$ \\ \hline
   $ g_i:$      & $\cdots$ & $g_2$   &  $g_1$ \\ \hline
\end{tabular}
\end{center}
which increase from  right to left.
\item  The terms in the sequence  $\lambda_i$ are independent of each other as well as  the terms in  the sequence $l-k+\lambda_{k}$.
  \end{enumerate}
\end{rmk}

In \cite{Wy09}, it is   pointed out   that  $\alpha_*$ is one more than  $\beta_*$  in the symbols  in the  $B_n, C_n$,  and $D_n$ theories.   Before showing it,  we give  some basic  facts. For the part $\lambda_{i}$ appearing even ($2m$) times in the partition,  according to the first step of the computation of symbol, we get the following sequence
  \begin{eqnarray}\label{SN}
  && (\lambda_{i}+l-i, \lambda_{i+1}+l-(i+1),\cdots,\lambda_{i+2m-1}+l-(i+2m-1) ) \\
  &&\qquad\qquad\qquad\qquad\qquad=(\lambda_{i}+l-i, \lambda_{i}+l-(i+1),\cdots,\lambda_{i}+l-(i+2m-1) ),\non
  \end{eqnarray}
 where $l$ is the length of the partition. With equal number of odd terms  and  even terms in  (\ref{SN}),  this sequence  leads to equal number of   $\alpha_*$ and $\beta_*$ in the symbol.  For the terms left in the sequence $l-k$,   the  number of the difference  between the  even terms and  odd terms is
\begin{equation}\label{SL}
  \frac{1-(-1)^l}{2}
\end{equation}
which is 1 if $l$ is odd and  0 if $l$ is even.

\begin{lemma}{\label{Lb}}
  The number of $\alpha_i$'s  is always one more than the number of $\beta_i$'s in the symbol of a partition $\lambda$ in the $B_n$  theory. The symbol  has the following form
\begin{equation}\label{s-1}
    \left(\ba{@{}c@{}c@{}c@{}c@{}c@{}c@{}c@{}c@{}c@{}} \alpha_1 &&  \alpha_2 &&\cdots  &&  \alpha_m  \\ & \beta_1 && \cdots && \beta_{m-1} && & \ea \right).
\end{equation}
\end{lemma}
\begin{proof} According to the definition of orthogonal partition, all the  even integers  appear an even number of times . By using formula(\ref{SN}),  all the even parts  lead to equal number of $\alpha_*$ and $\beta_*$.

Since the length $l$ of the partition   is odd\footnote{ It  will be proved in Proposition \ref{Pb}.} and the  number of the total boxes in $\lambda$ is odd,  the number of  odd parts is odd.  According to formula (\ref{SL}), for these odd parts $\lambda_i$, the number of odd terms $l-i+\lambda_i$  corresponding to $f_i$ which lead to $\alpha_*$ is one more than the number of even terms $l-i+\lambda_i$  corresponding to $g_i$ which lead to $\beta_*$.

Combing the results of the above  paragraphs, we draw the conclusion.
\end{proof}

 \begin{rmk} The total number of $\alpha_*$ and $\beta_*$ is $l$.
 \end{rmk}

\begin{lemma}{\label{Lc}}
The number of $\alpha_i$'s  is always one more than the number of $\beta_i$'s in the symbol of a partition $\lambda$ in the $C_n$  theory.
\end{lemma}
\begin{proof} When the length of the partition  is even, we  should append an extra 0 as the last part of the partition. So  the  length of the partition is odd.

According to the definition of symplectic partition, all the  odd integers appear an even number of times. According to formula(\ref{SN}),  all these odd parts  lead to equal number of $\alpha_*$ and $\beta_*$.

Since  $l$   is odd and the  number of the total boxes in $\lambda$ is even,  the number of the even parts in the partition is odd. According to formula(\ref{SL}), for these odd parts $\lambda_i$, the number of even terms $l-i+\lambda_i$  corresponding to $g_i$ which lead to $\alpha_*$ is one more than the number of odd terms $l-i+\lambda_i$  corresponding to $f_i$ which lead to $\beta_*$.

Combing the results of the above  paragraphs, we draw the conclusion.
\end{proof}

 \begin{rmk}
 If the length of the partition  $l$ is odd, the total number of  $\alpha_*$ and $\beta_*$ is $l$. If $l$ is even,   the total number of  $\alpha_*$ and $\beta_*$ is $l+1$  because of the  extra  0  appended.
  \end{rmk}

\begin{lemma}{\label{Ld}}
The number of $\alpha_i$'s  is always one more than the number of $\beta_i$'s in the symbol of a partition $\lambda$ in the $D_n$  theory.
\end{lemma}
\begin{proof}
For a partition in the $D_n$ theory, the length of the partition $l$ is even\footnote{It will be proved in Proposition \ref{Pd}.}and all the odd integers appear an even number times. According to formula(\ref{SN}),  these odd parts  lead to equal number of $\alpha_*$ and $\beta_*$.

Since $l$ is even and the  number of the total boxes in $\lambda$ is even,   the number of the even  parts in the partition is even. According to formula(\ref{SL}), for these even parts $\lambda_i$, the number of even terms $l-i+\lambda_i$  corresponding to $g_i$ which lead to $\alpha_*$ is equal to the number of odd terms $l-i+\lambda_i$  corresponding to $f_i$ which lead to $\beta_*$.

For rigid partition,  $\lambda_l=1$, so $l-l+\lambda_l=l-l+1$ corresponds to $2f_1+1$. Thus $f_1=0$ which means $\beta_1=f_1-1+1=0$. After $\beta_1$ is  omitted,  the number of $\alpha_i$'s in the top row of the symbol is always one more than the number of $\beta_i$'s on the bottom row.
\end{proof}

\begin{rmk}
Note that we delete one $\beta_1$ artifically and relabel $\beta_2 \rightarrow \beta_1$ etc. The total number of  $\alpha_*$ and $\beta_*$ is $l-1$.
\end{rmk}

\subsection{New definition of   symbol}
In this subsection, we propose equivalent definitions of symbols for the  $C_n$  and $D_n$ theories which  are  consistent with that for the $B_n$ theory  as   much as possible.
\begin{Def}\label{D2}
 The  symbols  of a partitions in the  $B_n $, $ C_n$,  and $D_n$ theories.
\begin{itemize}
  \item For the $B_n$ theory: first we add $l-k$  to the $k$~th part of the partition.
Next we arrange the odd parts of the sequence $l-k+\lambda_k$ in an increasing sequence $2f_i+1$  and arrange the even  parts in an increasing sequence $2g_i$.
Then we calculate terms
 \begin{equation*}
   \al_i = f_i-i+1\quad\quad\quad  \bet_i = g_i-i+1.
 \end{equation*}
 Finally we  write the {\it symbol} as
\begin{equation*}
  \left(\ba{@{}c@{}c@{}c@{}c@{}c@{}c@{}c@{}} \al_1 &&\al_2&&\al_3&& \cdots \\ &\bet_1 && \bet_2 && \cdots  & \ea \right).
\end{equation*}

  \item For the $C_n$ theory: \begin{description}
                                \item[1]If the length of partition is even,  compute the symbol as in the $B_n$ case, and then append an extra 0 on the left of the top row of the symbol.
                                \item[2]  If the length of the partition is odd, first append an extra 0 as the last part of the partition. Then compute the symbol as in the $B_n$ case. Finally,  we delete the 0 in the first entry of the bottom row of the symbol.
                              \end{description}
   \item For the $D_n$ theory: first append an extra 0 as the last part of the partition, and then compute the symbol as in the $B_n$ case. We  delete  two  0's in the first two entries of the bottom row of the symbol.
\end{itemize}

\begin{rmk}
   Compared  to the old definitions of symbol, $\alpha_*$ are all related to $f_*$  and $\beta_*$ are all related to $g_*$ in the new definitions of  symbol in the  $B_n$,  $C_n$,  and $D_n$ theories.
   \end{rmk}

\end{Def}

Before proving the equivalence between the new definition of symbol and  old one, we give several universal results. For the first step of the computation of  the symbol, we have the following  sequence
\begin{equation}\label{S}
  S=(l-1+\lambda_1, l-2+\lambda_2,  \cdots,  l-l+\lambda_l).
\end{equation}
 Arrange   the odd parts in a  decreasing sequence
\begin{equation}\label{So}
  (l-j_1+\lambda_{j_1}, l-j_2+\lambda_{j_2}, \cdots,l-j_{m}+ \lambda_{j_{m}})=(2f_{{m}}+1, 2f_{{m-1}}+1, \cdots, 2f_{1}+1).
\end{equation}
  And arrange    the even parts in a decreasing sequence
 \begin{equation}\label{Se}
  (l-i_1+\lambda_{i_1}, l-i_2+\lambda_{i_2}, \cdots, l-{i_n}+\lambda_{i_{n}})=(2g_{{n}}, 2g_{{n-1}}, \cdots, 2g_{1}).
 \end{equation}
Introduce the following notations
  \begin{equation}\label{AB}
    A_a=f_{a}-a+1, \quad\quad B_b=g_{b}-b+1,
  \end{equation}
where  $A_a$ corresponds to $f_{a}$ and   $B_b$ corresponds to $g_{b}$.  We write $A_*$ and $B_*$ as follows
\be\label{SA}
    \left(\ba{@{}c@{}c@{}c@{}c@{}c@{}c@{}c@{}c@{}c@{}} A_1 &&  A_2 &&\cdots  &&  A_m  \\ & B_1 && \cdots && B_{n} && & \ea \right).
\ee

If we append an extra 0 as the last part of the partition, the  sequence  corresponding to (\ref{S}) is
\begin{equation}\label{S1}
  S^{+}=((l+1)-1+\lambda_1, (l+1)-2+\lambda_2 , \cdots ,(l+1)-l+\lambda_l,0).
\end{equation}
Let $f^{'}_a$ and $g^{'}_a$ be the notations $f_a$ and $g_a$ for this sequence.  Arrange  the odd parts in a decreasing sequence
\begin{equation}\label{So1}
 (l+1-i_1+\lambda_{i_1}, l+1-i_2+\lambda_{i_2}, \cdots, l+1-{i_n}+\lambda_{i_{n}})=(2f{'}_{{n}}+1, 2f{'}_{{n-1}}+1, \cdots, 2f{'}_{1}+1).
\end{equation}
Comparing the left hand side of  formula (\ref{So1}) with that of  formula (\ref{Se}), we have
  $$(2f{'}_{{n}}+1, 2f{'}_{{n-1}}+1, \cdots, 2f{'}_{1}+1)=(2g_{{n}}+1, 2g_{{n-1}}+1, \cdots, 2g_{1}+1).$$
 Arrange  the even parts of sequence (\ref{S1}) in a decreasing sequence
  \begin{equation}\label{Se1}
   (l+1-j_1+\lambda_{j_1}, l+1-j_2+\lambda_{j_2}, \cdots,l+1-j_{m}+ \lambda_{j_{m}},0)=(2g{'}_{{m+1}}, 2g{'}_{{m}}, \cdots, 2g{'}_{2},0).
  \end{equation}
 Comparing the left hand side of  formula (\ref{Se1}) with that of  formula (\ref{So}), we have
  $$(2g{'}_{{m+1}}, 2g{'}_{{m}}, \cdots, 2g{'}_{2},0)=(2(f_{{m}}+1), 2(f_{{m-1}}+1), \cdots, 2(f_{1}+1),0).$$
 Let $A^{'}_a$ and $B^{'}_a$ be the notations $A_a$ and $B_a$ for the sequence  (\ref{S1})
  \begin{equation}\label{S1AB}
    A^{'}_b=f^{'}_{b}-b+1=g_{b}-b+1=B_b, \quad\quad B^{'}_a=g^{'}_{a}-a+1=(f_{{a-1}}+1)-a+1=A_{a-1}.
  \end{equation}
  Since the last term in  sequence (\ref{S1AB}) is 0, it corresponds to $B^{'}_1=A_0=0$
 which means
  \begin{equation}\label{SB}
   \left(\ba{@{}c@{}c@{}c@{}c@{}c@{}c@{}c@{}c@{}c@{}c@{}} && \,\,\,&&  A^{'}_1  &&   \cdots  && A^{'}_{n}& \\ & && B^{'}_1 && \cdots &&B^{'}_{m} && B^{'}_{m+1} \ea \right) =\left(\ba{@{}c@{}c@{}c@{}c@{}c@{}c@{}c@{}c@{}c@{}c@{}} && \,\,\,&&  B_1  &&   \cdots  && B_{n}& \\ &0 && A_1 && \cdots &&A_{m-1} && A_m \ea \right).
  \end{equation}

Similarly, we have the following proposition.
\begin{Pro}\label{c}
For a partition $\lambda$  in the $C_n$ theory, the new definition  of   symbol is equivalent to the old one.
\end{Pro}

  \begin{proof}  We need to consider two cases:

  \begin{itemize}
  \item  The length of  $\lambda$   is even. First, we compute the symbol by the new definition.  we have
$$ \alpha_i = f_i-i+1=A_i  \quad \quad  \beta_i =g_i-i+1=B_i.$$
Using formula(\ref{SA}), we get
\be\label{SAc}
  \left(\ba{@{}c@{}c@{}c@{}c@{}c@{}c@{}c@{}c@{}c@{}} A_1 &&  A_2 &&\cdots  &&  A_m  \\ & B_1 && \cdots && B_{n} && & \ea \right).
\ee
After appending  an extra 0 on the left of the  bottom row of formula(\ref{SAc}),  the symbol of $\lambda$   is
\be\label{newc}
    \left(\ba{@{}c@{}c@{}c@{}c@{}c@{}c@{}c@{}c@{}c@{}c@{}c@{}} 0 && A_1 &&  A_2 &&\cdots  &&  A_m  \\ &\,\,&& B_1 && \cdots && B_{n} && & \ea \right).
\ee
Next, we compute the  symbol by the old definition. After appending an extra 0 as the last part of the partition $\lambda$, using formula (\ref{SB}), we have
\begin{equation*}
    \left(\ba{@{}c@{}c@{}c@{}c@{}c@{}c@{}c@{}c@{}c@{}c@{}} && \,\,\,&&  B_1  &&   \cdots  && B_{n}& \\ &0 && A_1 && \cdots &&A_{m-1} && A_m \ea \right).
  \end{equation*}
Let $f^{'}_a$ and $g^{'}_a$ be the notations $f_a$ and $g_a$ for the sequence  (\ref{S1}) corresponding to the partition $\lambda\oplus 0$.
  According to the old definition, we have
  $$\alpha^{'}_i=g^{'}_i-i+1=B^{'}_i=A_{i-1},\quad\quad\beta^{'}_i= f^{'}_i-i+1=A^{'}_i=B_i.$$
  Thus the symbol of $\lambda$ is
\begin{equation*}
  \left(\ba{@{}c@{}c@{}c@{}c@{}c@{}c@{}c@{}c@{}c@{}c@{}c@{}} 0 && A_1 &&  A_2 &&\cdots  &&  A_m  \\ &\,\,&& B_1 && \cdots && B_{n} && & \ea \right)
\end{equation*}
which is equal to formula(\ref{newc}).

\item  The length of $\lambda$ is odd.   First, we compute the symbol by the old definition,   we have
$$ \alpha_i = g_i-i+1=B_i,  \quad \quad  \beta_i = f_i-i+1=A_i$$
which mean the symbol of $\lambda$ is
\begin{equation}\label{newcoo}
 \left(\ba{@{}c@{}c@{}c@{}c@{}c@{}c@{}c@{}c@{}c@{}c@{}} && \,\,\,&&  B_1  &&   \cdots  && B_{n}& \\ &\,\,&& A_1 && \cdots &&A_{m-1} && A_m \ea \right).
\end{equation}
Next,  we compute symbol by the new definition. After appending an extra 0 as the last part of partition $\lambda$, using formula (\ref{SB}), we have
\begin{equation*}
    \left(\ba{@{}c@{}c@{}c@{}c@{}c@{}c@{}c@{}c@{}c@{}c@{}} && \,\,\,&&  B_1  &&   \cdots  && B_{n}& \\ &0 && A_1 && \cdots &&A_{m-1} && A_m \ea \right).
  \end{equation*}
 After deleting  the 0 in the first entry of the bottom row of the symbol,  we have
  \begin{equation*}
 \left(\ba{@{}c@{}c@{}c@{}c@{}c@{}c@{}c@{}c@{}c@{}c@{}} && \,\,\,&&  B_1  &&   \cdots  && B_{n}& \\ &\,\,&& A_1 && \cdots &&A_{m-1} && A_m \ea \right)
\end{equation*}
which is equal to formula(\ref{newcoo}).
  \end{itemize}
\end{proof}

Similarly, we have the following proposition.
\begin{Pro}\label{d}
For a partition $\lambda$  in the $D_n$ theory, the new definition  of   symbol is equivalent to the odd one.
\end{Pro}

  \begin{proof}
 According to the old construction of the  symbol,   we have
$$ \alpha_i = g_i-i+1=B_i,  \quad \quad  \beta_i = f_i-i+1=A_i,$$
leading to
\begin{equation}\label{newdo}
 \left(\ba{@{}c@{}c@{}c@{}c@{}c@{}c@{}c@{}c@{}c@{}c@{}} && \,\,\,&&  B_1  &&   \cdots  && B_{n}& \\ &\,\,&& A_1 && \cdots &&A_{m-1} && A_m \ea \right).
\end{equation}
After  omitting  the term $\beta_1=A_1=0$, we have the symbol as follows
\begin{equation}\label{newdoo}
 \left(\ba{@{}c@{}c@{}c@{}c@{}c@{}c@{}c@{}c@{}c@{}c@{}} && \,\,\,&&  B_1  &&   \cdots  && B_{n}& \\ &\,\,&& A_2 && \cdots &&A_{m-1} && A_m \ea \right).
\end{equation}

Next, we compute the symbol by the new definition. After appending an extra 0 as the last part of partition $\lambda$, using formula (\ref{SB}), we have
\begin{equation*}
    \left(\ba{@{}c@{}c@{}c@{}c@{}c@{}c@{}c@{}c@{}c@{}c@{}} && \,\,\,&&  B_1  &&   \cdots  && B_{n}& \\ &0 && A_1 && \cdots &&A_{m-1} && A_m \ea \right).
  \end{equation*}
  After  deleting  two 0's in the first two terms  of  the bottom row of the  symbol,  we have
  \begin{equation*}
 \left(\ba{@{}c@{}c@{}c@{}c@{}c@{}c@{}c@{}c@{}c@{}c@{}} && \,\,\,&&  B_1  &&   \cdots  && B_{n}& \\ &\,\,&& A_2 && \cdots &&A_{m-1} && A_m \ea \right)
\end{equation*}
which is equal to formula(\ref{newdoo}).
\end{proof}

According to  proofs of the  above propositions \ref{c} and \ref{d},   operations of  deleting a 0 or appending a 0 in the last step of the calculation  of symbol are not essential. So we have the following equivalences of symbols.

 \textbf{Example:} For the  $C_{10}$ partition $3^2\,2^6\,1^2$, the length of the partition is even
\be
\left(\ba{@{}c@{}c@{}c@{}c@{}c@{}c@{}c@{}c@{}c@{}c@{}c@{}} 0&&0&&1&&1&&1&&1 \\ & 1 && 1 && 1 &&1&&2& \ea \right)\sim \left(\ba{@{}c@{}c@{}c@{}c@{}c@{}c@{}c@{}c@{}c@{}c@{}c@{}} &&0&&1&&1&&1&&1 \\ & 1 && 1 && 1 &&1&&2& \ea \right) \non
\ee

where  $\sim$ means that   two symbols are equivalent.

  \textbf{Example:} For the  $C_{12}$ partition $3^2\,2\,1^4$, the length of the partition is odd. We have
\be
\left(\ba{@{}c@{}c@{}c@{}c@{}c@{}c@{}c@{}c@{}c@{}c@{}c@{}} &&&&1&&1&&1&&1 \\ &  &&  && 0 &&0&&2& \ea\right)\sim
\left(\ba{@{}c@{}c@{}c@{}c@{}c@{}c@{}c@{}c@{}c@{}c@{}c@{}} &&&&1&&1&&1&&1 \\ & \,\,\, && 0 && 0 &&0&&2& \ea \right).\non
\ee

  \textbf{Example:} For the  $D_{10}$ partition $4^2\,3\,2^2\,1^5$,  we have
\be
\left(\ba{@{}c@{}c@{}c@{}c@{}c@{}c@{}c@{}c@{}c@{}c@{}c@{}c@{}c@{}} &&&&1&&1&&2&&2&&2 \\ &  &&  && 0&& 0 &&0&&2& \ea \right)\sim\left(\ba{@{}c@{}c@{}c@{}c@{}c@{}c@{}c@{}c@{}c@{}c@{}c@{}c@{}c@{}} &&\,\,\, &&1&&1&&2&&2&&2 \\ &  0&&  0&& 0&& 0 &&0&&2& \ea \right).\non
\ee

The above  equivalences  imply the  following definition of symbol.
\begin{Def}\label{Dn} The  symbol  of a partition in the  $B_n $, $ C_n$,  and $D_n$ theories.
\begin{itemize}
    \item For the $B_n$ theory: first we add $l-k$  to the $k$th part of the partition.
Next we arrange the odd parts and  the even  parts of the sequence $l-k+\lambda_k$ in an increasing sequence $2f_i+1$  and  an increasing sequence $2g_i$ respectively.
Then we calculate terms
 \begin{equation*}
   \al_i = f_i-i+1\quad\quad\quad  \bet_i = g_i-i+1.
 \end{equation*}
 Finally we  write the {\it symbol} as
\begin{equation*}
  \left(\ba{@{}c@{}c@{}c@{}c@{}c@{}c@{}c@{}c@{}c@{}} \alpha_1 &&  \alpha_2 &&\cdots  &&  \alpha_m  \\ & \beta_1 && \cdots && \beta_{m-1} && & \ea \right).
\end{equation*}

  \item For the $C_n$ theory: if the length of partition is even,  compute the symbol as the $B_n$ case.  If the length of the partition is odd, first append an extra 0 as the last part of the partition. Then compute the symbol as in the $B_n$ case. The symbol has the following form
  \be\label{Dc}
   \left(\ba{@{}c@{}c@{}c@{}c@{}c@{}c@{}c@{}c@{}c@{}c@{}c@{}} &&\alpha_1 &&\alpha_2 &&\cdots &&\alpha_{m-1}&&\alpha_{m}\\ & \beta_1 && \beta_2 && \cdots &&\beta_{m-1} && \beta_{m} & \ea \right).
\ee
\item For the $D_n$ theory: first append an extra 0 as the last part of the partition, then compute the symbol as in the $B_n$ case.
 The symbol has the following form
\be\label{Dd}
\left(\ba{@{}c@{}c@{}c@{}c@{}c@{}c@{}c@{}c@{}c@{}c@{}c@{}c@{}c@{}} &&\,\,\, &&\alpha_1&&\alpha_2&&\cdots&&\alpha_{m-1}&&\alpha_m \\ &  \beta_1&&  \beta_2&& \beta_3&& \cdots &&\beta_m&&\beta_{m+1}& \ea \right).
\ee
\end{itemize}
\end{Def}

 \begin{rmk}
 The symbol  of a partition in  $B_n$, $ C_n$, and $D_n$ theories has the following form
 \end{rmk}
 \be\label{Dt}
    \boxed{\left(\ba{@{}c@{}c@{}c@{}c@{}c@{}c@{}c@{}c@{}c@{}} \alpha_1 &&  \alpha_2 &&\cdots  &&  \alpha_m  \\ & \beta_1 && \cdots && \beta_{m+t} && & \ea \right)},
\ee
where  $m=(l+(1-(-1)^l)/2)/2$ and $l$ is the length of the partition.  $t=-1$ for the $B_n$ theory,  $t=0$  for the $C_n$ theory, and  $t=1$ for the  $D_n$ theory.
%After omitting terms $\beta_1, \cdots\beta_{m+t}$ in the symbol (\ref{Dt}).
\textit{The above  definitions of symbol will be the starting  point of the proofs in the following sections.}

\subsection{Formal operations of a partition}
We discuss  several formal operations of a partition which will be used in the construction of symbol in the next section.

The first operation is to append  an extra 0 as the last part of the  partition $\lambda$. Then we compute the  symbol of the new partition $\lambda\oplus0$. In fact, this operation leads to formula(\ref{SN}). Note that for a partition in the $D_n$ theory and for a partition with the first row is odd in  the $C_n$ theory we need append an extra 0 as the last part of the partition. The partition appended an extra 0 is notated as $\lambda^0$. And $l^0$ is the length of the partition $\lambda^0$.
\begin{lemma}\label{Lem} For a partition in the  $B_n$, $C_n$,  and $D_n$ theories, the symbol is
 \be\label{ss}
    \left(\ba{@{}c@{}c@{}c@{}c@{}c@{}c@{}c@{}c@{}c@{}} \alpha_1 &&  \alpha_2 &&\cdots  &&  \alpha_m  \\ & \beta_1 && \cdots && \beta_{m+t} && & \ea \right)
\ee
where  $t=-1$ for the  $B_n$ theory,  $t=0$  for the $C_n$ theory,  and  $t=1$ for the $D_n$ theory. $\lambda^0$ denote    the partition appended an extra 0 if needed for the computation of symbol.  Then the symbol of $\lambda^{'}=\lambda^0\oplus0^1$ computed as a partition  in the $B_n$ theory is
 \begin{equation}\label{ls}
    \left(\ba{@{}c@{}c@{}c@{}c@{}c@{}c@{}c@{}c@{}c@{}c@{}} && \,\,\,&&  \beta_1  &&   \cdots  && \beta_{m+t}& \\ &0 && \alpha_1 && \cdots &&\alpha_{m-1} && \alpha_m \ea \right).
  \end{equation}
And the symbol of $\lambda^0\oplus0^2$ computed as a partition  in the $B_n$ theory is
    \begin{equation}\label{l0}
    \left(\ba{@{}c@{}c@{}c@{}c@{}c@{}c@{}c@{}c@{}c@{}c@{}c@{}}0&& \alpha_1 &&  \alpha_2 &&\cdots  &&  \alpha_m  \\ &  0 && \beta_1 && \cdots && \beta_{m+t} && & \ea \right).
\end{equation}
\end{lemma}

  \begin{proof} Formula(\ref{ls}) is the result of formula(\ref{SB}).
Applying formula (\ref{ls}) twice, we get formula(\ref{l0}).
\end{proof}
\begin{rmk}\begin{enumerate}
                     \item After appending an extra 0, we have $\lambda^0_{l+1}=0$ which means
                     $$l^0+1-(l^0+1)+\lambda_{l^0+1}=0.$$ The superscript '0' means it is an index in the partition $\lambda^0$.  So we get $g^{'}_1=0$ which means $\beta^{'}_1=g^{'}_1-1+1=0$. $\beta^{'}_1=0$ is in the  first  entry  on the left side of the bottom row in  the symbol  (\ref{ls}).
                     \item After  appending an extra 0, the length of the partition is increased by one. The even terms of the sequence $l^0-k$ become the  odd terms of the  sequence $l^0+1-k$. And  the odd terms of the sequence $l^0-k$ become the  even terms of the sequence $l^0+1-k$.  The bottom row and  top row in symbol  reverse roles by comparing formula(\ref{l0}) with  (\ref{ls}).
                   \end{enumerate}
\end{rmk}

Next, we introduce another two formal  operations: the first one is to add a row $0^{l^0+2a}$ to a partition $\lambda$. And the second one is to add   a column $\lambda_0(=\lambda_1)$ to a partition $\lambda^0$. Now we discuss what happen to  the symbol of the result partitions.
\begin{lemma}\label{Lem-v}
For a partition $\lambda$ in the  $B_n$, $C_n$, and $D_n$ theories, the symbol is
 \be
    \left(\ba{@{}c@{}c@{}c@{}c@{}c@{}c@{}c@{}c@{}c@{}} \alpha_1 &&  \alpha_2 &&\cdots  &&  \alpha_m  \\ & \beta_1 && \cdots && \beta_{m+t} && & \ea \right)\non
\ee
where  $t=-1$ for the  $B_n$ theory,  $t=0$  for the $C_n$ theory, and  $t=1$ for the $D_n$ theory.
$\lambda^0$ denote   the partition appended an extra 0  if needed for the computation of symbol.  Then the symbol of the partition $\lambda^{'}=\lambda^0+0^{l^0+2a}$   is
\begin{equation}\label{ls-v}
\Bigg(\!\!\!\ba{c} \overbrace{0\;\cdots\; 0}^{a}\;\;\alpha_1\;\;\alpha_2 \cdots \alpha_m   \\
 \underbrace{0 \;\cdots \;0}_a \;\;\beta_1 \cdots  \beta_{m+t}\ \ea \Bigg),
\end{equation}
which is computed as a partition  in the $B_n$ theory
If we append an extra $\lambda_0 (=\lambda_1)$ as the first part of the partition $\lambda$, then the symbol of the partition  $\lambda^{''}=\lambda_0\oplus\lambda^0$ computed as a partition  in the $B_n$ theory is
    \begin{equation}\label{l0-v}
    \left(\ba{@{}c@{}c@{}c@{}c@{}c@{}c@{}c@{}c@{}c@{}c@{}c@{}} \alpha_1 &&  \alpha_2 &&\cdots  &&  \alpha_m  && \square^t \\ & \beta_1 && \cdots && \beta_{m+t} && \square_b && & \ea \right)
    \end{equation}
  where   $(\square^t,\square^b)=(\alpha_{m+1},\emptyset)$ or $(\square^t,\square^b)=(\emptyset, \beta_{m+t+1})$.

\end{lemma}

  \begin{proof} After adding a row $0^{l^0+2a}$ under the first row of $\lambda^0$,  nothing  changed for the first $l^0$ terms in the sequence  $l^0-k+\lambda_k$.   Using formula(\ref{l0}) $m$ times,  we get  formula(\ref{ls-v}).

         Because of  the extra part $\lambda_0$,  there will be  an extra term in the symbol of the partition $\lambda^{''}$ compared with that of $\lambda^0$. It   appear at the end of the  row of the symbol, since $(l^0+1)-0+\lambda_0$ is the biggest one  of  the sequence  $(l^0+1)-k+\lambda_k$.
\end{proof}

\section{Construction of  symbol}
Our main strategy to construct symbol  is to sum the contribution of each row of the partition. The addition of symbol is defined by writing  the symbols right adjusted and simply add the entries that are 'in the same place'.  Unfortunately,  it is not natural  to consider the contribution   of  each odd row independently.  However, it is natural to consider the contribution of   two odd rows together. In the first subsection, we introduce three lemmas which are crucial  for constructing  of symbol. Then using these lemmas, we give simple rules to  calculate symbol.  Finally, we give an uniform  closed formula of symbol  for the partitions in the $B_n$, $C_n$, and $D_n$ theories.
\subsection{Three  Lemmas }
First, we analyse the contribution to symbol of  even rows in  the partition. Added an even row, the result  partition  is still   in the same theory. \textit{The  processes of the computation of symbol are the same in the $B_n$, $C_n$, and $D_n$ theories excepting   an extra 0 appended.}
\begin{lemma}\label{Lsy}
  For the partition $\lambda(\lambda_1=\cdots=\lambda_a>\cdots),(a>2b)$, the symbol   has the following form
\begin{equation}\label{sy}
 \left(\ba{@{}c@{}c@{}c@{}c@{}c@{}c@{}c@{}c@{}c@{}} \cdots&& \cdots && \alpha_{m-b+1} &&\cdots && \alpha_m \\   & \cdots&& \beta_{m-b+1+t} && \cdots && \beta_{m+t} &\ea \right).
 \end{equation}
  $l$ is the length of the partition  including the extra  0 appended   if needed.  If $l+\lambda_1$ is odd, the symbol of $\lambda^{'}=\lambda+1^{2b}$  has the following form
\begin{equation}\label{ao}
\left(\ba{@{}c@{}c@{}c@{}c@{}c@{}c@{}c@{}c@{}c@{}} \cdots&& \cdots && \alpha_{m-b+1}+1 &&\cdots && \alpha_m+1 \\   & \cdots&& \beta_{m-b+1+t} && \cdots && \beta_{m+t} &\ea \right).
\end{equation}
If $l+\lambda_1$ is even, the symbol of $\lambda^{'}=\lambda+ 1^{2b}$  has the following form
\begin{equation}\label{ae}
 \left(\ba{@{}c@{}c@{}c@{}c@{}c@{}c@{}c@{}c@{}c@{}} \cdots&& \cdots && \alpha_{m-b+1} &&\cdots && \alpha_m \\   & \cdots&& \beta_{m-b+1+t}+1 && \cdots && \beta_{m+t}+1 &\ea \right).
\end{equation}
\end{lemma}

  \begin{proof} Formula (\ref{sy}) is the result of Definition \ref{Dn}. For the first step of the  computation of symbol, we have the following sequence

   \begin{equation}\label{sy1}
    (l-1+\lambda_1,\cdots,l-2b+\lambda_{2b}, l-(2b+1)+\lambda_{2b+1},\cdots,0+\lambda_{l}).
  \end{equation}
Since $\lambda_1=\cdots=\lambda_{2b}$, the  first $2b$ terms of the sequence  (\ref{sy1}) are successive
\begin{equation}\label{sy2b}
  (l-1+\lambda_1,\cdots,l-2b+\lambda_{1}).
\end{equation}
According to Definition \ref{Dn}, all  even terms in the sequence (\ref{sy2b}) can be written in a decreasing sequence    as follows
\begin{equation*}
  (2g_{m+t},2g_{m+t-1},\cdots,2g_{m+t-b+1})
\end{equation*}
which are one to one correspondence with terms
$$(\beta_{m+t} ,\beta_{m+t-1} ,\cdots , \beta_{m+t-b+1}).$$

 The odd terms in the sequence (\ref{sy2b})  can be written in a decreasing sequence    as follows
\begin{equation*}
  (2f_{m}+1,2f_{m-2}+1,\cdots,2f_{m-b+1}+1)
\end{equation*}
which are one to one correspondence with terms
$$(\alpha_{m} ,\alpha_{m-1} ,\cdots , \alpha_{m-b+1}).$$
After adding an even row $1^{2b}$ to $\lambda$,  the sequence corresponding to (\ref{sy2b}) for the partition $\lambda^{'}=\lambda+1^{2b}$ is
\begin{equation}\label{sy2bn}
  (l-1+\lambda_1+1,\cdots,l-2b+\lambda_{1}+1).
\end{equation}

Let $g^{'}_{i}$ and  $f^{'}_{i}$  be the notations $g_i$ and  $f_i$ for the sequence (\ref{sy2bn}).
\begin{itemize}
  \item If $l+\lambda_1$ is odd, the  even terms in the sequence (\ref{sy2b})  can be rewritten as follows
\begin{equation}\label{o-even}
  (l-1+\lambda_1,l-3+\lambda_1,\cdots,l-(2b-1)+\lambda_{1})=(2g_{m+t},2g_{m+t-1},\cdots,2g_{m+t-b}).
\end{equation}
While the even terms in the sequence (\ref{sy2bn})  are
\begin{equation}\label{o-evenn}
  (l-1+\lambda_1,l-3+\lambda_1,\cdots,l-(2b-1)+\lambda_{1})=(2g^{'}_{m+t},2g^{'}_{m+t-1},\cdots,2g^{'}_{m+t-b}).
\end{equation}
Because of   no differences between   the  sequences (\ref{o-evenn}) and (\ref{o-even}), we have
\begin{equation}\label{lll1}
  (\beta^{'}_{m+t} ,\beta^{'}_{m+t-1} ,\cdots , \beta^{'}_{m+t-b})=(\beta_{m+t} ,\beta_{m+t-1} ,\cdots , \beta_{m+t-b}).
\end{equation}
The odd terms in the sequence (\ref{sy2b}) can be rewritten  as follows
\begin{equation}\label{o-old}
  (l-2+\lambda_1,l-4+\lambda_1,\cdots,l-(2b)+\lambda_{1})=(2f_{m}+1,2f_{m-2}+1,\cdots,2f_{m-b}+1)
\end{equation}
While the odd terms in the sequence (\ref{sy2bn})  are
\begin{equation}\label{o-oldn}
  (l-1+\lambda_1+1,l-3+\lambda_1+1,\cdots,l-(2b-1)+\lambda_{1}+1)=(2f^{'}_{m}+1,2f^{'}_{m-2}+1,\cdots,2f^{'}_{m-b}+1)
\end{equation}
Comparing the sequences (\ref{o-oldn}) with (\ref{o-old}), using $\alpha^{'}_i=f^{'}_i-i+1=(f_i+1)-i+1$, we have
\begin{equation}\label{lll2}
(\alpha^{'}_{m} ,\alpha^{'}_{m-1} ,\cdots , \alpha^{'}_{m-b})=(\alpha_{m}+1 ,\alpha_{m-1}+1 ,\cdots , \alpha_{m-b}+1).
\end{equation}
Combing  formulas (\ref{lll1}) and  (\ref{lll2}), we have formula(\ref{ao}).

\item If $l+\lambda_1$ is even, the   even terms in the sequence (\ref{sy2b}) of the  partition $\lambda$  can be rewritten as follows
\begin{equation}\label{e-even}
   (l-2+\lambda_1,l-4+\lambda_1,\cdots,l-(2b)+\lambda_{1})=(2g_{m+t},2g_{m+t-1},\cdots,2g_{m+t-b}).
\end{equation}

While the even terms in the sequence (\ref{sy2bn})  of  $\lambda^{'}$ are
\begin{equation}\label{e-evenn}
  (l-1+\lambda_1+1,l-3+\lambda_1+1,\cdots,l-(2b-1)+\lambda_{1}+1)=(2g^{'}_{m+t},2g^{'}_{m+t-1},\cdots,2g^{'}_{m+t-b}).
\end{equation}
Comparing  sequences (\ref{e-even}) with (\ref{e-evenn}),  using $\beta^{'}_i=g^{'}_i-i+1=(g_i+1)-i+1$, we have
\begin{equation}\label{bbss}
  (\beta^{'}_{m+t} ,\beta^{'}_{m+t-1} ,\cdots , \beta^{'}_{m+t-b})=(\beta_{m+t}+1 ,\beta_{m+t-1}+1 ,\cdots , \beta_{m+t-b}+1).
\end{equation}

These odd terms in the sequence (\ref{sy2b}) of  $\lambda$ are
\begin{equation}\label{e-old}
 (l-1+\lambda_1,l-3+\lambda_1,\cdots,l-(2b-1)+\lambda_{1})=(2f_{m}+1,2f_{m-2}+1,\cdots,2f_{m-b}+1)
\end{equation}
While odd terms in the  sequence (\ref{sy2bn}) of $\lambda^{'}$ are
\begin{equation}\label{e-oldn}
  (l-1+\lambda_1,l-3+\lambda_1,\cdots,l-(2b-1)+\lambda_{1})=(2f^{'}_{m}+1,2f^{'}_{m-2}+1,\cdots,2f^{'}_{m-b}+1)
\end{equation}
 Since nothings are changed in the sequence (\ref{e-oldn}) compared  with  (\ref{e-old}),  we have
$$(\alpha^{'}_{m} ,\alpha^{'}_{m-1} ,\cdots , \alpha^{'}_{m-b})=(\alpha_{m},\alpha_{m-1},\cdots , \alpha_{m-b}).$$
Combing  formulas (\ref{bbss}) and  (\ref{e-oldn}), we have  formula (\ref{ae}).
\end{itemize}
\end{proof}

\begin{rmk}
 The  entries in the same  row of the symbol are increased by one   from right to left. And the number of the changed entries  is half length  of the even row added to the partition.
                      %\item According to the proof,  conclusions are independent of the parity of  $a$.
\end{rmk}

\begin{lemma}\label{Lsysy}
  For a partition $\lambda (\lambda_1=\cdots=\lambda_a>\cdots)$, the last $b,(a>2b),$ terms in the top and bottom  rows of   the symbol  has the following form
\begin{equation}\label{sysy}
 \left(\ba{@{}c@{}c@{}c@{}c@{}c@{}c@{}c@{}c@{}c@{}} \cdots&& \cdots && \alpha_m &&\cdots && \alpha_m \\   & \cdots&& \beta_{m+t} && \cdots && \beta_{m+t} &\ea \right),
 \end{equation}
 where $t=-1$  for the  $B_n$ theory , $t=0$ for the  $C_n$ theory, and $t=1$ for the $D_n$ theory.
 $l$ is the length of the partition  including the extra  0 appended   if needed. If $l+\lambda_1$ is odd, $\beta_{m+t}=\alpha_m-t+1.$
    If $l+\lambda_1$ is even, $\beta_{m+t}=\alpha_m-t.$

\end{lemma}

\begin{proof} For the first step of the  computation of   symbol, we have the following sequence
   \begin{equation}\label{sysy1}
    (l-1+\lambda_1,\cdots,l-2b+\lambda_{2b}, l-(2b+1)+\lambda_{2b+1},\cdots,0+\lambda_{l}).
  \end{equation}
Since $\lambda_1=\cdots=\lambda_{2b}$, the  first $2b$ terms of the sequence  (\ref{sysy1}) are successive
\begin{equation}\label{sysy2b}
  (l-1+\lambda_1,\cdots,l-2b+\lambda_{1}).
\end{equation}
According to Definition \ref{Dn}, the even terms in the  sequence (\ref{sysy2b})  can be rewritten  as
\begin{equation*}
  (2g_{m+t},2g_{m+t-1},\cdots,2g_{m+t-b+1})
\end{equation*}
which  are one to one correspondence with terms
$$(\beta_{m+t} ,\beta_{m+t-1} ,\cdots , \beta_{m+t-b+1}).$$
Since $g_i=g_{i-1}+1$ for $i\geq m+t-b$, we have
$$\beta_i=\beta_{i-1},\quad i\geq m+t-b.$$
The odd terms in the sequence (\ref{sysy2b})  can be written as
\begin{equation*}
  (2f_{m}+1,2f_{m-2}+1,\cdots,2f_{m-b+1}+1)
\end{equation*}
which are one to one correspondence with terms
$$(\alpha_{m} ,\alpha_{m-1} ,\cdots , \alpha_{m-b+1}).$$
Since $f_i=f_{i-1}$ for $i\geq m-b$, we have
$$\alpha_i=\alpha_{i-1}, \quad i\geq m-b. $$

\begin{itemize}
    \item If $l+\lambda_1$ is odd, so the largest number  $l-1+\lambda_1$  in the  sequence (\ref{sysy1})  is even, thus corresponding to $\beta_{m+t}$. We have $g_{m+t}=f_m+1$
        which mean $$\beta_{m+t}=g_{m+t}-(m+t)+1=f_m+1-(m+t)+1=f_m-m+1-t+1=\alpha_m-t+1.$$
    \item If $l+\lambda_1$ is even, so the largest number  $l-1+\lambda_1$  in the sequence (\ref{sysy1}) is odd, thus corresponding to $\alpha_{m}$. We have $g_{m+t}=f_m$
        which mean $$\beta_{m+t}=g_{m+t}-(m+t)+1=f_m-(m+t)+1=f_m-m+1-t=\alpha_m-t.$$
  \end{itemize}
\end{proof}

% Append two odd rows in the $B_n$ theory
Next, we discuss what happen to the symbol of a partition when added an odd row. We can reduce the proof to the case of Lemma \ref{Lsy} by introducing an virtual column before the first part of the partition. We find that it is more natural to consider the contributions of two odd rows together.
\begin{lemma}\label{Lo}
For a partition $\lambda(\lambda_1=\cdots=\lambda_a>\cdots)$, $(a>2b+1)$,  the symbol  has the following form
\begin{equation}\label{oL1}
\left(\ba{@{}c@{}c@{}c@{}c@{}c@{}c@{}c@{}c@{}c@{}c@{}c@{}c@{}c@{}} &&  \cdots&& \alpha_{m-b} && \alpha_{m-b+1} &&\cdots && \alpha_{m-1}&& \alpha_m \\    & \cdots&& \beta_{m+t-b}&& \beta_{m+t-b+1} && \cdots &&\beta_{m+t-1} && \beta_{m+t} &\ea \right).
\end{equation}
 $l$ is the length of the partition  including the extra  0 appended  if necessary.
If $l+\lambda_1$ is odd,  the symbol of the partition $\lambda^{'}=\lambda+1^{2b+1}$ is
\begin{equation}\label{oL2}
\left(\ba{@{}c@{}c@{}c@{}c@{}c@{}c@{}c@{}c@{}c@{}c@{}c@{}c@{}c@{}c@{}c@{}} &&  \cdots&& \alpha_{m-b} && \alpha_{m-b+1} &&\cdots && \alpha_{m-1}&&\alpha_m && \alpha_{m}\\    & \cdots&& \beta_{m+t-b}&&\beta_{m+t-b+1}+1 && \cdots &&\beta_{m+t-1}+1 &&\quad\quad && \quad\quad &\ea \right).
\end{equation}
If $l+\lambda_1$ is even,  the symbol of the partition $\lambda^{'}=\lambda+1^{2b+1}$ is
\begin{equation}\label{oL3}
\left(\ba{@{}c@{}c@{}c@{}c@{}c@{}c@{}c@{}c@{}c@{}c@{}c@{}c@{}c@{}c@{}c@{}} &&  \cdots&& \alpha_{m-b}+1 && \alpha_{m-b+1}+1 &&\cdots && \alpha_{m-1}+1&& \,\,\quad && \,\,\quad\\    & \cdots&& \beta_{m+t-b}&&\beta_{m+t-b+1} && \cdots &&\beta_{m+t-1} && \beta_{m+t}&& \beta_{m+t}  &\ea \right).
\end{equation}
\end{lemma}

  \begin{proof} To use Lemma \ref{Lsy},  we append a virtual row $\lambda_0 (=\lambda_1 )$  before the first part $\lambda_1$, thus forming a new partition  $\lambda^{V}(\lambda_0=\lambda_1=\cdots=\lambda_a>\cdots)$ with length of $l+1$.
   For the first step of   computation of symbol, we have the following sequence
   \begin{equation}\label{osy1}
    (\underline{l+1-1+\lambda_0},l-1+\lambda_1,\cdots,l-2b+\lambda_{2b}, l-(2b+1)+\lambda_{2b+1},\cdots,0+\lambda_{l}).
  \end{equation}
The underline   indicate that   it is an auxiliary term.
%$\underline{l-0+\lambda_0} l+1-1+\lambda_0$
After adding a row $1^{2b+2}$ to the partition,  we have the following sequence
\begin{equation}\label{osy2}
    (\underline{l+1-1+\lambda_0+1},l-1+\lambda_1+1,\cdots,l-2b+\lambda_{2b}+1+1, l-(2b+1)+\lambda_{2b+1}+1,\cdots,0+\lambda_{l}).
  \end{equation}
 %The first term is an auxiliary term.

\begin{itemize}
  \item If $l+\lambda_1$  is odd,   $l+1-1+\lambda_0$ in the sequence (\ref{osy1})  corresponds to $\alpha_{m+1}$.
Then the  symbol of the partition $\lambda^{V}$ is
  \begin{equation}\label{oL2bo1}
\left(\ba{@{}c@{}c@{}c@{}c@{}c@{}c@{}c@{}c@{}c@{}c@{}c@{}c@{}c@{}c@{}c@{}} &&  \cdots&& \alpha_{m-b} && \alpha_{m-b+1}&&\cdots && \alpha_{m-1}&& \alpha_m && \alpha_{m+1}\\    & \cdots&& \beta_{m+t-b}&&\beta_{m+t-b+1} && \cdots &&\beta_{m+t-1} && \beta_{m+t}&& \quad\quad &\ea \right).
\end{equation}
According to Lemma \ref{Lsysy}, we have $\alpha_{m}=\alpha_{m+1}$.
   For the partition $\lambda^{V}+1^{2b+2}$,  $l+1+\lambda_0$ is even.   Using Lemma \ref{Lsy}, we have
 \begin{equation}\label{oL2bo2}
\left(\ba{@{}c@{}c@{}c@{}c@{}c@{}c@{}c@{}c@{}c@{}c@{}c@{}c@{}c@{}c@{}c@{}} &&  \cdots&& \alpha_{m-b} && \alpha_{m-b+1} &&\cdots && \alpha_{m-1}&&\alpha_m && \alpha_{m}\\    & \cdots&& \beta_{m+t-b}&&\beta_{m+t-b+1}+1 && \cdots &&\beta_{m+t-1}+1 &&\beta_{m+t}+1&& \quad\quad &\ea \right).
\end{equation}
Since the auxiliary term $\underline{l-0+\lambda_0+1}$ in the sequence (\ref{osy2}) is even, it corresponds to the last one on the bottom row of the symbol which is  $\beta_{m+t}+1$. After omitting  this term, the symbol of $\lambda^{'}=\lambda+1^{2b+1}$ is
  \begin{equation}\label{oL2bo3}
\left(\ba{@{}c@{}c@{}c@{}c@{}c@{}c@{}c@{}c@{}c@{}c@{}c@{}c@{}c@{}c@{}c@{}} &&  \cdots&& \alpha_{m-b} && \alpha_{m-b+1} &&\cdots && \alpha_{m-1}&&\alpha_m && \alpha_{m}\\    & \cdots&& \beta_{m+t-b}&&\beta_{m+t-b+1}+1 && \cdots &&\beta_{m+t-1}+1 &&\quad\quad&& \quad\quad &\ea \right)
\end{equation}
which is formula(\ref{oL2}).

\item $l+\lambda_1$  is even,   $l+1-1+\lambda_0$ in the sequence (\ref{osy1})    corresponds to $\beta_{m+t+1}$.
Then the  symbol of the partition $\lambda^{V}$ is
  \begin{equation*}
\left(\ba{@{}c@{}c@{}c@{}c@{}c@{}c@{}c@{}c@{}c@{}c@{}c@{}c@{}c@{}c@{}c@{}} &&  \cdots&& \alpha_{m-b} && \alpha_{m-b+1}&&\cdots && \alpha_{m-1}&& \alpha_m && \quad\quad\\    & \cdots&& \beta_{m+t-b}&&\beta_{m+t-b+1} && \cdots &&\beta_{m+t-1} && \beta_{m+t}&& \beta_{m+t+1} &\ea \right).
\end{equation*}
According to Lemma \ref{Lsysy}, we have $\beta_{m+t}=\beta_{m+t+1}$.
For the partition $\lambda^{V}+1^{2b+2}$,  $l+1+\lambda_0$ is odd.   Using Lemma \ref{Lsy}, we have
 \begin{equation}\label{oL2be1}
\left(\ba{@{}c@{}c@{}c@{}c@{}c@{}c@{}c@{}c@{}c@{}c@{}c@{}c@{}c@{}c@{}c@{}} &&  \cdots&& \alpha_{m-b}+1 && \alpha_{m-b+1}+1&&\cdots && \alpha_{m-1}+1&& \alpha_m+1 && \quad\quad\\    & \cdots&& \beta_{m+t-b}&&\beta_{m+t-b+1} && \cdots &&\beta_{m+t-1} && \beta_{m+t}&& \beta_{m+t} &\ea \right).
\end{equation}
Since the auxiliary term $\underline{l-0+\lambda_0+1}$ in the sequence (\ref{osy2}) is odd, it corresponds to $\alpha_{m}$ which should be omitted in the end. So the symbol of $\lambda^{'}=\lambda+1^{2b+1}$ is
  \begin{equation}\label{oL2be3}
\left(\ba{@{}c@{}c@{}c@{}c@{}c@{}c@{}c@{}c@{}c@{}c@{}c@{}c@{}c@{}c@{}c@{}} &&  \cdots&& \alpha_{m-b}+1 && \alpha_{m-b+1}+1&&\cdots && \alpha_{m-1}+1&& \quad\quad && \quad\quad\\    & \cdots&& \beta_{m+t-b}&&\beta_{m+t-b+1} && \cdots &&\beta_{m+t-1} && \beta_{m+t}&& \beta_{m+t} &\ea \right)
\end{equation}
which is  formula(\ref{oL3}).
\end{itemize}
\end{proof}
\begin{rmk}
                Note that we omit the entries  in the symbols (\ref{oL2bo2}), (\ref{oL2be1}),  which  correspond  to the auxiliary term in the sequence.
                     % \item According to the proof,  conclusions are independent of the parity of  the length   $a$.
\end{rmk}

\subsection{Symbol of rigid partition in the $B_n$ theory}
The following fact \cite{Wy09} is useful  for  studying   the structure of symbol, so we give the proof in detail.
\begin{Pro}{\label{Pb}}
The longest row in a rigid  $B_n$ partition always contains an odd number of boxes. And the following two rows of the first row are either both of odd length or both of even length.  This pairwise pattern then continues. If the Young tableau has an even number of rows, the row of shortest length has to be even.
\end{Pro}
\begin{proof}  Even integers appear an even number of times  for a partition in the $B_n$ theory.  So the sum of all odd integers is odd, which  implies that the  number of odd integers is odd.  And the length of the partition which is the sum of the number of odd integers and  even integers  is odd. So the longest row   contains an odd number of boxes.

If the following two rows of the first row are of different parities, then the difference of the number of boxes between these two rows is odd. It imply that part  2 appears odd number of times  in the partition, which is a  contradiction. So  the following two rows are either both of odd length or both of even length. In the same way, we can prove the   next two rows are of  the same parities. This pairwise rows continues.

The number of total  boxes  of a pairwise rows is even. If the Young tableau has an even number of rows, then the number of the total boxes of  the first row and shortest  row is odd.  Since the longest row  contains an odd number of boxes, the first row is even.
\end{proof}
\begin{rmk}
 If the last row is odd, the number of rows of the partition is odd.
\end{rmk}

Using the above characteristics of the partitions in the $B_n$ theory,  we can refine  Lemma \ref{Lsysy} as follows.
\begin{lemma}\label{Lbe}
Let  $\lambda (\lambda_1=\cdots=\lambda_a>\cdots> \lambda_{l})$ is a partition in the $B_n$ theory and  the last two rows have the  same parity.
If $a>2b$,  the symbol of $\lambda$  has the following form
\begin{equation}\label{Lb1}
 \left(\ba{@{}c@{}c@{}c@{}c@{}c@{}c@{}c@{}c@{}c@{}} \cdots&& \cdots && \alpha_{m-b+1} &&\cdots && \alpha_m \\   & \cdots&& \beta_{m-b} && \cdots && \beta_{m-1} &\ea \right)=\left(\ba{@{}c@{}c@{}c@{}c@{}c@{}c@{}c@{}c@{}c@{}} \cdots&& \cdots && \alpha_{m} &&\cdots && \alpha_m \\   & \cdots&& \alpha_{m} +1 && \cdots && \alpha_m +1 &\ea \right).
\end{equation}
The symbol of the partition $\lambda^{'}=\lambda+1^{2b}$ is
\begin{equation}\label{Lb2}
\left(\ba{@{}c@{}c@{}c@{}c@{}c@{}c@{}c@{}c@{}c@{}} \cdots&& \cdots && \alpha_{m-b+1} &&\cdots && \alpha_m \\   & \cdots&& \beta_{m-b}+1 && \cdots && \beta_{m-1}+1 &\ea \right).
\end{equation}
The symbol of the partition $\lambda^{''}=\lambda^{'}+1^{2c}$, $b>c$, is
\begin{equation}\label{Lb3}
  \left(\ba{@{}c@{}c@{}c@{}c@{}c@{}c@{}c@{}c@{}c@{}c@{}c@{}c@{}c@{}} \cdots&& \cdots && \alpha_{m-b+1} &&\cdots && \alpha_{m-c+1}+1 &&  \cdots&& \alpha_m +1\\   & \cdots&& \alpha_{m-b} +1 &&\cdots && \alpha_{m-b+1} +1 &&  \cdots && \beta_{m-1} +1 &\ea \right).
\end{equation}
\end{lemma}

  \begin{proof} According to Lemma \ref{Lsysy}, we get  formula (\ref{Lb1}).  According to Proposition \ref{Pb},    $\lambda_1$ is odd which means $l+\lambda_1$ is even, thus formula (\ref{Lb2}) is the result of  formula (\ref{ae}).  For partition $\lambda^{'}$,  $l+\lambda^{'}_1=l+\lambda_1+1$  is odd,  thus  formula (\ref{Lb3}) is the result of   formula (\ref{ao}).
\end{proof}

\begin{rmk}
\begin{enumerate}
                      \item Compared  formulas (\ref{Lb2})  with  (\ref{Lb1}),   the   contribution to symbol of the row $1^{2b}$  is
\be\label{sbe1}
 \Bigg(\!\!\!\ba{c}0\;\;0\cdots 0 \;\; 0\cdots 0\;\; \overbrace{ 0\cdots 0}^{c} \\
\;\;\;0\cdots 0 \;\; \underbrace{1\cdots 1\;1\cdots 1}_{b} \ \ea \Bigg).
\ee
                      \item  Compared formulas  (\ref{Lb3})  with  (\ref{Lb2}), the   contribution to symbol of the row $1^{2c}$  is
                          \be\label{sbe2}
 \Bigg(\!\!\!\ba{c}0\;\;0\cdots 0 \;\; 0\cdots 0\;\; \overbrace{ 1\cdots 1}^{c} \\
\;\;\;0\cdots 0 \;\; \underbrace{0\cdots 0\;0\cdots 0}_{b} \ \ea \Bigg).
\ee

                      \item These patterns (\ref{sbe1}) and (\ref{sbe2}) of the contributions to symbol for the    even rows of a pairwise rows   will continue.
\end{enumerate}\end{rmk}

If  a   partition  in the $B_n$ theory is  added an even row, the new partition is in the same theory  by Proposition \ref{Pb}.
However, if added  an odd row,  the result partition is not in the same theory.  If  added  two odd rows with the same parity,   the new partition is also in the same theory.

We regard  adding an odd  row to a partition   as a formal   operation.  And  Lemma \ref{Lo} can be rewritten   as follows.
\begin{lemma}\label{Lbo}
$\lambda (\lambda_1=\cdots=\lambda_a>\cdots\geq\lambda_{l})$ is a partition in the $B_n$ theory and the last two rows have the  same parities.
If $a>2b+1$,  the symbol of $\lambda$  has the following form
\begin{equation}\label{oLb1}
\left(\ba{@{}c@{}c@{}c@{}c@{}c@{}c@{}c@{}c@{}c@{}c@{}c@{}c@{}c@{}} \cdots&&  \cdots&& \alpha_{m-b} && \alpha_{m-b+1} &&\cdots && \alpha_{m-1}&& \alpha_m \\    & \cdots&&  \alpha_{m-b} +1&& \alpha_{m-b+1} +1 && \cdots && \alpha_{m-1} +1 && \alpha_{m} +1 &\ea \right).
\end{equation}
And the symbol of the partition $\lambda^{'}=\lambda+1^{2b+1}$ is
\begin{equation}\label{oLb2}
\left(\ba{@{}c@{}c@{}c@{}c@{}c@{}c@{}c@{}c@{}c@{}c@{}c@{}c@{}c@{}c@{}c@{}} \cdots&&  \cdots&& \alpha_{m-b}+1 && \alpha_{m-b+1}+1 &&\cdots && \alpha_{m-1}+1&& \,\,\quad && \,\,\quad\\    & \cdots&&  \alpha_{m-b} +1&& \alpha_{m-b+1} +1 && \cdots && \alpha_{m-1} +1 && \alpha_{m} +1 && \alpha_m +1 &\ea \right).
\end{equation}
And the symbol of the partition $\lambda^{''}=\lambda+1^{2b+1}+1^{2c+1},\,(c<b)$  is
\begin{equation}\label{oLb3}
  \left(\ba{@{}c@{}c@{}c@{}c@{}c@{}c@{}c@{}c@{}c@{}c@{}c@{}c@{}c@{}} \cdots&& \cdots && \alpha_{m-b}+1 &&\cdots && \alpha_{m-c+1}+1 &&  \cdots&& \alpha_m +1\\   & \cdots&& \alpha_{m-b} +1 &&\cdots && \alpha_{m-c+1} +2 &&  \cdots && \alpha_{m} +2 &\ea \right).
\end{equation}
\end{lemma}

  \begin{proof} Formula (\ref{oLb1}) is the result of Lemma \ref{Lsysy}.  Since the last two rows have the same parity,  $\lambda_1$ is odd which means $l+\lambda_1$ is even and thus  formula (\ref{oLb2}) is the result of formula (\ref{oL3}).   For the partition $\lambda^{'}$, $l+\lambda^{'}_1=l+\lambda_1+1$ is odd, and thus  formula (\ref{oLb3}) is the result of  formula (\ref{oL2}).
\end{proof}

\begin{rmk}
 Compared  formulas (\ref{oLb3})  with  (\ref{Lb1}), the  contribution to symbol of the two  rows $1^{2b+1}+1^{2c+1}$     is
       \be\label{sbo}
 \Bigg(\!\!\!\ba{c}0\;\;0\cdots 0 \;\; \overbrace{1\cdots 1\;\; 1\cdots1}^{b+1} \\
\;\;\;0\cdots 0 \;\;0\cdots0\;\underbrace{1\cdots 1}_{c} \ \ea \Bigg)=\Bigg(\!\!\!\ba{c}0\;\;0\cdots 0 \;\; \overbrace{1\cdots 1\;\; 1\cdots1}^{b+1} \\
\;\;\;0\cdots 0 \;\;0\cdots0\;\underbrace{0\cdots 0}_{c} \ \ea \Bigg)+\Bigg(\!\!\!\ba{c}0\;\;0\cdots 0 \;\; \overbrace{0\cdots 0\;\; 0\cdots0}^{b+1} \\
\;\;\;0\cdots 0 \;\;0\cdots0\;\underbrace{1\cdots 1}_{c} \ \ea \Bigg).
\ee
Formally, the first term on the right side  can be regarded  as the contribution of the row $1^{2b+1}$ by comparing  formulas (\ref{oLb2})  with (\ref{Lb1}). And  the second  term  can be regarded  as the contribution  of the row $1^{2c+1}$ by comparing  formulas (\ref{oLb3})  with  (\ref{Lb2}).
  \end{rmk}

\subsubsection{Closed formula of  symbol}\label{tb}
 We summary the  remarks of Lemmas \ref{Lbe} and  \ref{Lbo} as follows\footnote{Here, we discuss the contribution to symbol of odd rows formally. We will find that this  method is reasonable after  introducing the map $X_S$ (\ref{XS})in Section 5. }

\begin{center}
\begin{tabular}{|c|c|c|c|}\hline
\multicolumn{4}{|c|}{ Contribution to  symbol of the $i$\,th row of a partition }\\ \hline
Parity of   row & Parity of $i$ &  $L$ &  Contribution to symbol \\ \hline
odd & even  & $\frac{1}{2}(\sum^{m}_{k=i}n_k+1)$ & $\Bigg(\!\!\!\ba{c}0 \;\; 0\cdots \overbrace{ 1\;\; 1\cdots1}^{L} \\
\;\;\;0\cdots 0\;\; 0\cdots 0 \ \ea \Bigg)$   \\ \hline
even & odd    & $\frac{1}{2}(\sum^{m}_{k=i}n_k)$ & $\Bigg(\!\!\!\ba{c}0 \;\; 0\cdots \overbrace{ 1\;\; 1\cdots1}^{L} \\
\;\;\;0\cdots 0\;\; 0\cdots 0 \ \ea \Bigg)$  \\ \hline
even & even    & $\frac{1}{2}(\sum^{m}_{k=i}n_k)$ &  $\Bigg(\!\!\!\ba{c}0 \;\; 0\cdots 0\;\; 0 \cdots 0 \\
\;\;\;0\cdots \underbrace{1 \;\;1\cdots 1}_{L} \ \ea \Bigg)$  \\ \hline
odd & odd     & $\frac{1}{2}(\sum^{m}_{k=i}n_k-1)$  &  $\Bigg(\!\!\!\ba{c}0 \;\; 0\cdots 0\;\; 0 \cdots 0 \\
\;\;\;0\cdots \underbrace{1\; \;1\cdots 1}_{L} \ \ea \Bigg)$  \\ \hline
\end{tabular}
\end{center}

To compute symbol,  besides these rules in the above table,   the contribution to symbol of the first row of a partition  should be calculated as an initial condition. Although,  the  contribution of the first row  can be calculated   directly, there is  another method which is particularly revealing. According to formula (\ref{ls-v}),  the symbol of the first row  $1^l$  is equal to the symbol of the partition $1^l+0^l$.  By using formula (\ref{oL2bo2}), its contribution   is
\begin{equation}\label{F1}
\Bigg(\!\!\!\ba{c}\overbrace{0\;\;0\cdots \cdots 0}^{(l+1)/2}  \\
 \;\;\;\underbrace{1\cdots \cdots 1}_{(l-1)/2}\ \ea \Bigg)
\end{equation}
which is consistent with the remark  of Lemma \ref{Lbo}. We draw the conclusion that the contribution to symbol of the first row can be seen as $i=1$ case in the above table.

We can describe the above table concisely as follows
\begin{flushleft}
\textbf{ Rules ($B_n$)}: Formally, a row in a partition      contribute   1 in succession  from right to left  in  the same row of the symbol. The contribution  of adjoining rows with the different parities occupy the same row of  symbol,  otherwise occupy  the other  row. The number of 1  contributed by even row is   one half length  of the  row.  The number of 1  contributed by the first odd row of a pairwise rows  is   one half of the number which is  the length of the  row plus  one.  The number of 1  contributed by the second odd row of a pairwise rows  is   one half of the number which is   the length of the row minus one.
\end{flushleft}

Compared to original  definition of symbol, the above  rules are easy to remember  and convenient to operator on.
According to the rules, we give a closed formula of the symbol of a rigid partition in the $B_n$ theory.
\begin{Pro}\label{Fb}
For a partition $\lambda=m^{n_m}{(m-1)}^{n_{m-1}}\cdots{1}^{n_1}$ in the $B_n$ theory, we introduce the  following notations
$$\Delta^{T}_i=\frac{1}{2}(\sum^{m}_{k=i}n_k+\frac{1+(-1)^{i+1}}{2}),\quad P^{T}_i=\frac{1+\pi_i}{2}$$
where the superscript $T$ indicate it is related to the top row of the symbol and
$$\pi_i=(-1)^{\sum^{m}_{k=i}n_k}\cdot(-1)^{i}.$$
Other parallel notations
$$\Delta^{B}_i=\frac{1}{2}(\sum^{m}_{k=i}n_k+\frac{1+(-1)^{i}}{2}),\quad  P^{B}_i=\frac{1-\pi_i}{2}$$
where the superscript $B$ indicates  it is related to the bottom row of the symbol and $P^{b}_i$ is a projection operator similar to  $P^{t}_i$.
Notating    the symbol of $\lambda$ as $\sigma(\lambda)$, then we have
\begin{equation}\label{FbFbFb}
\sigma(\lambda)=\sum^m_{i=1}\Bigg\{ \Bigg(\!\!\!\ba{c}0\;\;0\cdots 0\;\; \overbrace{1\cdots 1}^{P^{T}_i \Delta^{T}_i}  \\
\;\;\;\underbrace{0\cdots 0 \;\;0 \cdots 0}_{\frac{l-1}{2}}\ \ea \Bigg)
+
 \Bigg(\!\!\!\ba{c}\overbrace{0\;\;0\cdots 0\;\; 0\cdots 0}^{\frac{l+1}{2}}  \\
\;\;\;0\cdots 0 \;\;\underbrace{1\cdots 1}_{P^{B}_i \Delta^{B}_i}\ \ea \Bigg)\Bigg\}.
\end{equation}
\end{Pro}

\subsection{Symbol of   rigid partition in  the  $C_n$ theory}\label{tc}
The contents of this subsection and next subsection are parallel to the previous subsection. First, we prove the following proposition
\begin{Pro}{\label{Pc}}
The longest two rows in a rigid $C_n$ partition both contain either  an even or an odd number  of boxes.  This pairwise rows then continues. If the Young tableau has an odd number of rows, the row of shortest length has contain an even number of boxes.
\end{Pro}
\begin{proof} For a partition in the $C_n$ theory,  odd integers appear an even number of times.
If the first two rows  are of different parities, then  the  difference of the number of boxes between them will be odd. It  imply that part  '1'  appear odd number of times   which is a  contradiction.  So we have proved that the lengths of the  first two rows are either both odd  or  even.  In the same way, we can prove  that the   next two rows will be of  the same parity. This pairwise rows then  continues.
\end{proof}

Using the above proposition, we can refine  Lemmas \ref{Lsy}  and \ref{Lsysy} as follows
\begin{lemma}\label{Lce}
For a partition $\lambda (\lambda_1=\cdots=\lambda_a>\cdots\geq\lambda_{l})$ in the $C_n$ theory,  the last two rows have the  same parity. $l$ is the length of the partition including the extra 0  if necessary.
If  $a>b$,  the symbol of $\lambda$  has the following form
\begin{equation}\label{Lc1}
 \left(\ba{@{}c@{}c@{}c@{}c@{}c@{}c@{}c@{}c@{}c@{}} \cdots&& \cdots && \alpha_{m-b+1} &&\cdots && \alpha_m \\   & \cdots&& \beta_{m-b+1} && \cdots && \beta_{m} &\ea \right)=\left(\ba{@{}c@{}c@{}c@{}c@{}c@{}c@{}c@{}c@{}c@{}} \cdots&& \cdots && \alpha_{m} &&\cdots && \alpha_m \\   & \cdots&& \alpha_{m}  && \cdots && \alpha_m  &\ea \right).
\end{equation}
And the symbol of the partition $\lambda^{'}=\lambda+1^{2b}$ is
\begin{equation}\label{Lc2}
\left(\ba{@{}c@{}c@{}c@{}c@{}c@{}c@{}c@{}c@{}c@{}} \cdots&& \cdots && \alpha_{m-b+1} &&\cdots && \alpha_m \\   & \cdots&& \beta_{m-b+1}+1 && \cdots && \beta_{m}+1 &\ea \right).
\end{equation}
And the symbol of the partition $\lambda^{''}=\lambda^{'}+1^{2c}$ is
\begin{equation}\label{Lc3}
  \left(\ba{@{}c@{}c@{}c@{}c@{}c@{}c@{}c@{}c@{}c@{}c@{}c@{}c@{}c@{}} \cdots&& \cdots && \alpha_{m-b+1} &&\cdots && \alpha_{m-c+1}+1 &&  \cdots&& \alpha_m +1\\   & \cdots&& \alpha_{m-b+1} +1 &&\cdots && \alpha_{m-c+1} +1 &&  \cdots && \beta_{m} +1 &\ea \right).
\end{equation}
\end{lemma}

  \begin{proof} According to  remark  of  Definition \ref{Dn}, $l$ is even.  According to  Lemma \ref{Lsysy}, we get  formula(\ref{Lc1}).  According to Proposition \ref{Pc},    $\lambda_1$ is even  which means $l+\lambda_1$ is even, and thus formula (\ref{Lc3}) is the result of formula(\ref{ae}).  Similarly,  for partition $\lambda^{'}$,  $l+\lambda^{'}_1=l+\lambda_1+1$  is odd, thus formula(\ref{Lb2}) is the result of formula (\ref{ao}).
\end{proof}

\begin{rmk}\begin{enumerate}
                      \item Compared formulas (\ref{Lc2})  with (\ref{Lc1}),    the  contribution to symbol of row $1^{2b}$  is
\be\label{sce1}
 \Bigg(\!\!\!\ba{c}0\;\;0\cdots 0 \;\; 0\cdots 0\;\; \overbrace{ 0\cdots 0}^{c} \\
\;\;\;0\cdots 0 \;\; \underbrace{1\cdots 1\;1\cdots 1}_{b} \ \ea \Bigg)
\ee
                      \item  Compared formulas (\ref{Lc3})  with (\ref{Lc2}), the  contribution  to symbol  of  row $1^{2c}$  is
                          \be\label{sce2}
 \Bigg(\!\!\!\ba{c}0\;\;0\cdots 0 \;\; 0\cdots 0\;\; \overbrace{ 1\cdots 1}^{c} \\
\;\;\;0\cdots 0 \;\; \underbrace{0\cdots 0\;0\cdots 0}_{b} \ \ea \Bigg)
\ee

                    \end{enumerate}
\end{rmk}

Next, we discuss the contributions to symbol of  odd rows  formally.
\begin{lemma}\label{Lco}
For a partition $\lambda (\lambda_1=\cdots=\lambda_a>\cdots\geq\lambda_{l})$ in the $C_n$ theory,  the last two rows have the  same parity. $l$ is the length of the partition including the extra 0 in the last part of partition if necessary.
If $a>2b+1$,  the symbol of $\lambda$  has the following form
\begin{equation}\label{oLc1}
\left(\ba{@{}c@{}c@{}c@{}c@{}c@{}c@{}c@{}c@{}c@{}c@{}c@{}c@{}c@{}} \cdots&&  \cdots&& \alpha_{m-b} && \alpha_{m-b+1} &&\cdots && \alpha_{m-1}&& \alpha_m \\    & \cdots&&  \alpha_{m-b} && \alpha_{m-b+1}  && \cdots && \alpha_{m-1}  && \alpha_{m}  &\ea \right).
\end{equation}
And the symbol of the partition $\lambda^{'}=\lambda+1^{2b+1}$ is
\begin{equation}\label{oLc2}
\left(\ba{@{}c@{}c@{}c@{}c@{}c@{}c@{}c@{}c@{}c@{}c@{}c@{}c@{}c@{}c@{}c@{}} \cdots&&  \cdots&& \alpha_{m-b}+1 && \alpha_{m-b+1}+1 &&\cdots && \alpha_{m-1}+1&& \,\,\quad && \,\,\quad\\    & \cdots&&  \alpha_{m-b} && \alpha_{m-b+1}  && \cdots && \alpha_{m-1} && \alpha_{m}  && \alpha_m &\ea \right).
\end{equation}
And the symbol of the partition $\lambda^{''}=\lambda+1^{2b+1}+1^{2c+1}, \,(c<b),$  is
\begin{equation}\label{oLc3}
  \left(\ba{@{}c@{}c@{}c@{}c@{}c@{}c@{}c@{}c@{}c@{}c@{}c@{}c@{}c@{}} \cdots&& \cdots && \alpha_{m-b}+1 &&\cdots && \alpha_{m-c+1}+1 &&  \cdots&& \alpha_m +1\\   & \cdots&& \alpha_{m-b}  &&\cdots && \alpha_{m-c+1} +1 &&  \cdots && \alpha_{m} +1 &\ea \right).
\end{equation}
\end{lemma}

  \begin{proof} Formula(\ref{oLc1}) is the result of Lemma \ref{Lsysy}. Since $\lambda_1$ is even  which means $l+\lambda_1$ is even, thus formula(\ref{oLc2}) is the result of formula(\ref{oL3}).   For partition $\lambda^{'}$, $l+\lambda^{'}_1=l+\lambda_1+1$ is odd, thus the   formula (\ref{oLc3}) is the result of formula(\ref{oL2}).
\end{proof}

\begin{rmk}
   Compared formulas (\ref{oLc3})  with  (\ref{Lc1}), the contribution to symbol of $1^{2b+1}+1^{2c+1}$ is
       \be\label{sco}
 \Bigg(\!\!\!\ba{c}0\;\;0\cdots 0 \;\; \overbrace{1\cdots 1\;\; 1\cdots1}^{b+1} \\
\;\;\;0\cdots 0 \;\;0\cdots0\;\underbrace{1\cdots 1}_{c} \ \ea \Bigg)=\Bigg(\!\!\!\ba{c}0\;\;0\cdots 0 \;\; \overbrace{1\cdots 1\;\; 1\cdots1}^{b+1} \\
\;\;\;0\cdots 0 \;\;0\cdots0\;\underbrace{0\cdots 0}_{c} \ \ea \Bigg)+\Bigg(\!\!\!\ba{c}0\;\;0\cdots 0 \;\; \overbrace{0\cdots 0\;\; 0\cdots0}^{b+1} \\
\;\;\;0\cdots 0 \;\;0\cdots0\;\underbrace{1\cdots 1}_{c} \ \ea \Bigg).
\ee
Formally, the first term on the right side  can be regarded  as the contribution  of row $1^{2b+1}$ by comparing formulas (\ref{oLc2})  with (\ref{Lc1}). And  the second  term  can be regarded  as the contribution  of row $1^{2c+1}$ by  comparing formulas (\ref{oLc3})  with (\ref{Lc2}).
  \end{rmk}

 We  can summary the above results  as the  following stable
\begin{center}
\begin{tabular}{|c|c|c|c|}\hline
\multicolumn{4}{|c|}{ Contribution to  symbol of the $i$\,th row}\\ \hline
Parity of the length of   row & Parity of $i+1$  & $L$ & Contribution to symbol  \\ \hline
odd & even& $\frac{1}{2}(\sum^{m}_{k=i}n_k+1)$   & $\Bigg(\!\!\!\ba{c}0 \;\; 0\cdots \overbrace{ 1\;\; 1\cdots1}^{L} \\
\;\;\;0\cdots 0\;\; 0\cdots 0 \ \ea \Bigg)$  \\ \hline
even & odd  & $\frac{1}{2}(\sum^{m}_{k=i}n_k)$  & $\Bigg(\!\!\!\ba{c}0 \;\; 0\cdots \overbrace{ 1\;\; 1\cdots1}^{L} \\
\;\;\;0\cdots 0\;\; 0\cdots 0 \ \ea \Bigg)$   \\ \hline
even & even  & $\frac{1}{2}(\sum^{m}_{k=i}n_k)$ &  $\Bigg(\!\!\!\ba{c}0 \;\; 0\cdots 0\;\; 0 \cdots 0 \\
\;\;\;0\cdots \underbrace{1 \;\;1\cdots 1}_{L} \ \ea \Bigg)$    \\ \hline
odd & odd& $\frac{1}{2}(\sum^{m}_{k=i}n_k-1)$  &  $\Bigg(\!\!\!\ba{c}0 \;\; 0\cdots 0\;\; 0 \cdots 0 \\
\;\;\;0\cdots \underbrace{1\; \;1\cdots 1}_{L} \ \ea \Bigg)$       \\ \hline
\end{tabular}
\end{center}
It is easy to check that the contributions to symbol of the first two rows  are consistent with this table. Note that this table is the same with the table in the $B_n$ case except checking  of  the  parity of the index  $i+1$.

\subsection{Symbol of   rigid partition in  the $D_n$ theory}\label{td}
 We can prove the following proposition similar to the proof of Proposition \ref{Pb}.
\begin{Pro}{\label{Pd}}
The longest row in a rigid $D_n$ partition always contains an even number of boxes. And the following two rows are either both of even length or both of odd length. This pairwise rows then continue. If the Young tableau has an even number of rows the row of the shortest length has to be even.
\end{Pro}

Using the above proposition, we can refine  Lemma \ref{Lsy}  and \ref{Lsysy} as follows\begin{lemma}\label{Lde}
For a partition $\lambda (\lambda_1=\cdots=\lambda_a>\cdots\geq\lambda_{l})$ in the $D_n$ theory,  the last two rows have same parity. $l$ is the length of the partition including the extra 0 as the last part of partition if necessary.
If $a>b$,  the symbol of $\lambda$  has the  following form
\begin{equation}\label{Ld1}
 \left(\ba{@{}c@{}c@{}c@{}c@{}c@{}c@{}c@{}c@{}c@{}} \cdots&& \cdots && \alpha_{m-b+1} &&\cdots && \alpha_m \\   & \cdots&& \beta_{m-b+2} && \cdots && \beta_{m+1} &\ea \right)=\left(\ba{@{}c@{}c@{}c@{}c@{}c@{}c@{}c@{}c@{}c@{}} \cdots&& \cdots && \alpha_{m} &&\cdots && \alpha_m \\   & \cdots&& \alpha_{m}-1  && \cdots && \alpha_m-1  &\ea \right).
\end{equation}
And the symbol of the partition $\lambda^{'}=\lambda+1^{2b}$ is
\begin{equation}\label{Ld2}
\left(\ba{@{}c@{}c@{}c@{}c@{}c@{}c@{}c@{}c@{}c@{}} \cdots&& \cdots && \alpha_{m-b+1} &&\cdots && \alpha_m \\   & \cdots&& \beta_{m-b+2}+1 && \cdots && \beta_{m+1}+1 &\ea \right).
\end{equation}
And the symbol of the partition $\lambda^{''}=\lambda^{'}+1^{2c}$ $ (c<b)$ is
\begin{equation}\label{Ld3}
  \left(\ba{@{}c@{}c@{}c@{}c@{}c@{}c@{}c@{}c@{}c@{}c@{}c@{}c@{}c@{}} \cdots&& \cdots && \alpha_{m-b+1} &&\cdots && \alpha_{m-c+1}+1 &&  \cdots&& \alpha_m +1\\   & \cdots && \beta_{m-b+2}+1  &&\cdots && \beta_{m-c+2} +1 &&  \cdots && \beta_{m+1} +1 &\ea \right).
\end{equation}
\end{lemma}

  \begin{proof} According to the   remark  of  Definition \ref{Dn}, $l$ is odd.  Since $t=1$ in the $D_n$ theory,   combining  Lemma \ref{Lsysy}, we get formula(\ref{Ld1}).  According to Proposition \ref{Pd},    $\lambda_1$ is odd  which means $l+\lambda_1$ is even, thus formula (\ref{Ld3}) is the result of formula (\ref{ae}).  Similarly,  for partition $\lambda^{'}$,  $l+\lambda^{'}_1=l+\lambda_1+1$  is odd, thus  formula (\ref{Ld2}) is the result of formula (\ref{ao}).
\end{proof}

\begin{rmk} \begin{enumerate}
                      \item Compared formulas (\ref{Ld2})  with (\ref{Ld1}),   the contribution to symbol  of the row $1^{2b}$ is
\be\label{sde1}
 \Bigg(\!\!\!\ba{c}0\;\;0\cdots 0 \;\; \overbrace{ 1\cdots 1\;\; 1\cdots 1}^{b} \\
\;\;\;0\cdots 0 \;\;0\cdots 0\; \underbrace{0\cdots 0}_{c} \ \ea \Bigg)
\ee
                      \item  Compared formulas (\ref{Ld3})  with (\ref{Ld2}),  the contribution to symbol  of the row $1^{2c}$ is
 \be\label{sde2}
 \Bigg(\!\!\!\ba{c}0\;\;0\cdots 0 \;\; \overbrace{ 0\cdots 0\;\; 0\cdots 0}^{b} \\
\;\;\;0\cdots 0 \;\;0\cdots 0\; \underbrace{1\cdots 1}_{c} \ \ea \Bigg)
\ee

                    \end{enumerate}
\end{rmk}

Next, we discuss the contribution to symbol of the odd rows formally.
\begin{lemma}\label{Ldo}
For a partition $\lambda (\lambda_1=\cdots=\lambda_a>\cdots\geq\lambda_{l})$ in the $D_n$ theory,  the last two rows have the  same parity. $l$ is the length of the partition including the extra 0  if necessary for the computation of symbol.
If $a>2b+1$,  the symbol of $\lambda$  has the following form
\begin{equation}\label{oLd1}
\left(\ba{@{}c@{}c@{}c@{}c@{}c@{}c@{}c@{}c@{}c@{}c@{}c@{}c@{}c@{}} \cdots&&  \cdots&& \alpha_{m-b} && \alpha_{m-b+1} &&\cdots && \alpha_{m-1}&& \alpha_m \\    & \cdots&&  \alpha_{m-b}-1 && \alpha_{m-b+1}-1  && \cdots && \alpha_{m-1}-1  && \alpha_{m}-1  &\ea \right).
\end{equation}
And the symbol of the partition $\lambda^{'}=\lambda+1^{2b+1}$ is
\begin{equation}\label{oLd2}
\left(\ba{@{}c@{}c@{}c@{}c@{}c@{}c@{}c@{}c@{}c@{}c@{}c@{}c@{}c@{}c@{}c@{}} \cdots&&  \cdots&& \alpha_{m-b}+1 && \alpha_{m-b+1}+1 &&\cdots && \alpha_{m-1}+1&& \,\,\quad && \,\,\quad\\    & \cdots&&  \alpha_{m-b}-1 && \alpha_{m-b+1}-1  && \cdots && \alpha_{m-1}-1 && \alpha_{m}-1  && \alpha_m-1 &\ea \right).
\end{equation}
And the symbol of the partition $\lambda^{''}=\lambda+1^{2b+1}+1^{2c+1},\,(c<b)$  is
\begin{equation}\label{oLd3}
  \left(\ba{@{}c@{}c@{}c@{}c@{}c@{}c@{}c@{}c@{}c@{}c@{}c@{}c@{}c@{}} \cdots&& \cdots && \alpha_{m-b}+1 &&\cdots && \alpha_{m-c+1}+1 &&  \cdots&& \alpha_m +1\\   & \cdots&& \alpha_{m-b}  &&\cdots && \alpha_{m-c+1}  &&  \cdots && \alpha_{m}  &\ea \right).
\end{equation}
\end{lemma}

  \begin{proof} Formula(\ref{oLd1}) is the result of Lemma \ref{Lsysy}. Since $\lambda_1$ is odd , $l+\lambda_1$ is even, thus formula(\ref{oLd2}) is the result of formula (\ref{oL3}).   For partition $\lambda^{'}$, $l+\lambda^{'}_1=l+\lambda_1+1$ is odd, thus formula (\ref{oLd3}) is the result of formula(\ref{oL2}).
\end{proof}

\begin{rmk}
    Compared formulas (\ref{oLd3})  with (\ref{oLd1}), the contributions to symbol of rows $1^{2b+1}+1^{2c+1}$   is
       \be\label{sdo}
 \Bigg(\!\!\!\ba{c}0\;\;0\cdots 0 \;\; \overbrace{1\cdots 1\;\; 1\cdots1}^{b+1} \\
\;\;\;0\cdots 0 \;\;0\cdots0\;\underbrace{1\cdots 1}_{c} \ \ea \Bigg)=\Bigg(\!\!\!\ba{c}0\;\;0\cdots 0 \;\; \overbrace{1\cdots 1\;\; 1\cdots1}^{b+1} \\
\;\;\;0\cdots 0 \;\;0\cdots0\;\underbrace{0\cdots 0}_{c} \ \ea \Bigg)+\Bigg(\!\!\!\ba{c}0\;\;0\cdots 0 \;\; \overbrace{0\cdots 0\;\; 0\cdots0}^{b+1} \\
\;\;\;0\cdots 0 \;\;0\cdots0\;\underbrace{1\cdots 1}_{c} \ \ea \Bigg).
\ee
Formally, the first term on the right side of the  above formula can be regarded  as the contribution to symbol of the $1^{2b+1}$ by comparing formula(\ref{oLd2})  with  (\ref{Ld1}). And  the second  term  can be regarded  as the contribution  of  $1^{2c+1}$ by comparing formulas (\ref{oLd3})  with (\ref{oLd2}).
  \end{rmk}
  We summary these results  as the following table
\begin{center}
\begin{tabular}{|c|c|c|c|}\hline
\multicolumn{4}{|c|}{ Contribution to  symbol of the $i$\,th row of a partition }\\ \hline
Parity of length of  row & Parity of $i+2$  & $L$  & Contribution to symbol  \\ \hline
odd & even & $\frac{1}{2}(\sum^{m}_{k=i}n_k+1)$  & $\Bigg(\!\!\!\ba{c}0 \;\; 0\cdots \overbrace{ 1\;\; 1\cdots1}^{L} \\
\;\;\;0\cdots 0\;\; 0\cdots 0 \ \ea \Bigg)$  \\ \hline
even & odd & $\frac{1}{2}(\sum^{m}_{k=i}n_k)$ & $\Bigg(\!\!\!\ba{c}0 \;\; 0\cdots \overbrace{ 1\;\; 1\cdots1}^{L} \\
\;\;\;0\cdots 0\;\; 0\cdots 0 \ \ea \Bigg)$     \\ \hline
even & even  & $\frac{1}{2}(\sum^{m}_{k=i}n_k)$  &  $\Bigg(\!\!\!\ba{c}0 \;\; 0\cdots 0\;\; 0 \cdots 0 \\
\;\;\;0\cdots \underbrace{1 \;\;1\cdots 1}_{L} \ \ea \Bigg)$   \\ \hline
odd & odd    & $\frac{1}{2}(\sum^{m}_{k=i}n_k-1)$  &  $\Bigg(\!\!\!\ba{c}0 \;\; 0\cdots 0\;\; 0 \cdots 0 \\
\;\;\;0\cdots \underbrace{1\; \;1\cdots 1}_{L} \ \ea \Bigg)$   \\ \hline
\end{tabular}
\end{center}
It is easy to check that the contribution to symbol of the first two rows as initial condition are consistent with this table. Note that this table is the same  with the table in the $B_n$ case.

\subsection{Closed formula of  symbols  in the $B_n$, $C_n$, and $D_n$ theories}
With the  experiences learned  in the previous  subsections,  we find that there are  some common characteristics of the computational rules  of symbol in the  $B_n$, $C_n$, and $D_n$ theories.  From Lemmas \ref{Lsy}  and \ref{Lo}, the differences of the construction of  symbol in different  theories lie in the definition of the length of partition. From Definition  \ref{Dn},  we should append a 0 as the last part of the partition $\lambda$ in the $D_n$ theory   and   in the $C_n$ theory with  odd length of the partition.
 For a partition $\lambda=m^{n_m}{(m-1)}^{n_{m-1}}\cdots{1}^{n_1}$, we denote  $$n_0=\frac{1-(-1)^{\sum^{m}_{k=1}n_k}}{2}+t$$
where  $t=-1$ for $B_n$ theory, $t=0$ for $C_n$ theory, and  $t=1$ for $D_n$ theory.
Then the length $l$ of the partition  in  Lemmas \ref{Lsy} and  \ref{Lo}  is equal to the length of the following partition
  \begin{equation*}
    \lambda=m^{n_m}{(m-1)}^{n_{m-1}}\cdots{1}^{n_1}0^{n_0}.
  \end{equation*}
By using  Lemmas \ref{Lsy} and  \ref{Lo},   the contributions to symbol of  a pairwise rows  fit into the same forms as shown in the  formulas (\ref{sbe1}),(\ref{sbe2}),(\ref{sbo}),(\ref{sce1}),(\ref{sce2}),(\ref{sco}),(\ref{sde1}),(\ref{sde2}), and (\ref{sdo}).

Note that the tables of computational  rules of symbols in the  $B_n$, $C_n$,  and $D_n$ theories  are  exact in the  same  pattern.   We can describe  these tables   uniformly as follows\footnote{Here, we discuss the contribution to symbol of the odd rows formally. We will find that this formal method is reasonable after we introduce the maps $X_S (\ref{XS}), Y_S (\ref{YS})$ in Section 5. }
\begin{center}
\begin{tabular}{|c|c|c|c|}\hline
\multicolumn{4}{|c|}{ Contribution to  symbol of the $i$\,th row of a partition }\\ \hline
Parity of  row & Parity of $i+t+1$ & $L$ & Contribution to symbol   \\ \hline
odd & even & $\frac{1}{2}(\sum^{m}_{k=i}n_k+1)$  & $\Bigg(\!\!\!\ba{c}0 \;\; 0\cdots \overbrace{ 1\;\; 1\cdots1}^{L} \\
\;\;\;0\cdots 0\;\; 0\cdots 0 \ \ea \Bigg)$  \\ \hline
even & odd  & $\frac{1}{2}(\sum^{m}_{k=i}n_k)$ & $\Bigg(\!\!\!\ba{c}0 \;\; 0\cdots \overbrace{ 1\;\; 1\cdots1}^{L} \\
\;\;\;0\cdots 0\;\; 0\cdots 0 \ \ea \Bigg)$    \\ \hline
even & even & $\frac{1}{2}(\sum^{m}_{k=i}n_k)$  &  $\Bigg(\!\!\!\ba{c}0 \;\; 0\cdots 0\;\; 0 \cdots 0 \\
\;\;\;0\cdots \underbrace{1 \;\;1\cdots 1}_{L} \ \ea \Bigg)$    \\ \hline
odd & odd & $\frac{1}{2}(\sum^{m}_{k=i}n_k-1)$ &  $\Bigg(\!\!\!\ba{c}0 \;\; 0\cdots 0\;\; 0 \cdots 0 \\
\;\;\;0\cdots \underbrace{1\; \;1\cdots 1}_{L} \ \ea \Bigg)$       \\ \hline
\end{tabular}
\end{center}

We can describe the above table  concisely as the following rules.
\begin{flushleft}
  \textbf{ Rules }: Formally, a row of  a partition      contribute   '1' in succession  from right to left  in  the same row of the symbol. The contribution  of adjoining rows with  different parities occupy the same row of  symbol,  otherwise occupy  the other  row. The number of '1'  contributed by even row is   one half length  of the  row.  The number of '1'  contributed by the first odd row of a pairwise rows  is   one half of the number which is  the length of the  row plus  one.  The number of '1'  contributed by the second odd row of a pairwise rows  is   one half of the number which is   the length of the row minus one.

\end{flushleft}

According to  the above table, it is easy to   get a  closed formula of  symbol for the rigid partitions in the $B_n$, $C_n$, and $D_n$ theories.
\begin{Pro}\label{Fbcd}
For a partition $\lambda=m^{n_m}{(m-1)}^{n_{m-1}}\cdots{1}^{n_1}$, we introduce two notations
$$\Delta^{T}_i=\frac{1}{2}(\sum^{m}_{k=i}n_k+\frac{1+(-1)^{i+1}}{2}),\quad\quad  P^{T}_i=\frac{1+\pi_i}{2}$$
where the superscript $T$ indicates it is related to the top row of the symbol and
$$\pi_i=(-1)^{\sum^{m}_{k=i}n_k}\cdot(-1)^{i+1+t},$$
for $B_n(t=-1)$, $C_n(t=0)$, and $D_n(t=1)$ theories.
Other parallel notations
$$\Delta^{B}_i=\frac{1}{2}(\sum^{m}_{k=i}n_k+\frac{1+(-1)^{i}}{2}),\quad \quad  P^{B}_i=\frac{1-\pi_i}{2}$$
where the superscript $B$ indicates it is  related to the bottom row of the symbol and $P^{B}_i$ is a projection operator similar to  $P^{T}_i$.
Then the symbol $\sigma(\lambda)$   is
\begin{equation}\label{FbFb}
\sigma(\lambda)=\sum^m_{i=1} \Bigg\{ \Bigg(\!\!\!\ba{c}0\;\;0\cdots 0\;\; \overbrace{1\cdots 1}^{P^{T}_i \Delta^{T}_i}  \\
\;\;\;\underbrace{0\cdots 0 \;\;0 \cdots 0}_{l+t}\ \ea \Bigg)
+
 \Bigg(\!\!\!\ba{c}\overbrace{0\;\;0\cdots 0\;\; 0\cdots 0}^{l}  \\
\;\;\;0\cdots 0 \;\;\underbrace{1\cdots 1}_{P^{B}_i \Delta^{B}_i}\ \ea \Bigg) \Bigg\}
\end{equation}
with  $l=(m+(1-(-1)^m)/2)/2$.
\end{Pro}

\section{Applications}
As an  invariant of the partition, symbol can be used to study the $S$ duality pair of rigid unipotent partition \cite{GW08} \cite{Wy09}.  In \cite{ShO06},  we propose a simple rule to compute symbols of  partitions with  only even rows.  Symbol of a general partition in the $B_n, C_n$,  and $D_n$  theories can be calculated  by reducing a  partition into two partitions with only even rows  through two maps $X_S$ and $Y_S$. These two maps can be explained by the table in the previous section.
\subsection{Maps: $X_S$ and $Y_S$}
%There are partitions with only even rows  in the  $C_n$ and $D_n$ theories.  For partitions with only even rows, the rules in the previous section recover  the following  \textbf{Rule} proposed  in \cite{ShO06}
%\begin{flushleft}
%\textbf{Rule:} For a partition with only even rows in  the $C_n$ and $D_n$ theories, the  row with   $2m $ boxes   contribute $m $  '1'  in sequence  from  right to left  in the same   row of symbol,  while other entries of the symbol are  '0'. The contribution to symbol of the adjoining even rows  are formed in the same  way, excepting  the   '1's  occupy another row of symbol.
%\end{flushleft}

A general partition $\lambda$ in the  $B_n, C_n$, and $D_n$  theories can be reduced to two partitions with only even rows  by the  maps $X_S$ and $Y_S$ \cite{Wy09}.  The symbol of the partition   $\lambda$  is the sum of    the symbols of the result   partitions.  Before introducing these two maps, we prove the   following  proposition.
\begin{Pro}\label{Dual}
For a partition in the $B_n$ theory or a partition with only  even  rows in the $C_n$ theory,  we have $\alpha_i\leq \beta_{i+t}$.  For a partition in the $D_n$  theory or a partition with only  odd rows  in the $C_n$ theory, we have $\alpha_i\geq \beta_{i+t}$.
\end{Pro}

  \begin{proof} For a partition in the $B_n$ theory, according to formula(\ref{F1}),   the symbol of  the first row with $2m+1$ boxes is

  $$\Bigg(\!\!\!\ba{c}0 \;\; 0\cdots 0\;\; 0 \cdots 0 \\
\;\;\;\underbrace{1\cdots 1\; \;1\cdots 1}_{m} \ \ea \Bigg).$$
So, we draw the conclusion   for the first row.

According to Proposition \ref{Pb},  the following two rows of the first row are either both odd  or  even length. By using formula(\ref{Lb2}), the contributions to symbol of  two even rows $1^{2b}+1^{2c}, (b>c), $ are
\be
 \Bigg(\!\!\!\ba{c}0\;\;0\cdots 0 \;\; 0\cdots 0\;\; \overbrace{ 1\cdots 1}^{c} \\
\;\;\;0\cdots 0 \;\; \underbrace{1\cdots 1\;1\cdots 1}_{b} \ \ea \Bigg).
\ee
By using formula(\ref{oLb2}), the contributions to symbol of  two odd rows $1^{2b+1}+1^{2c+1}, (b>c),$ are
\be
 \Bigg(\!\!\!\ba{c}0\;\;0\cdots 0 \;\; \overbrace{1\cdots 1\;\;  1\cdots 1}^{b+1} \\
\;\;\;0\cdots 0 \;\; 0\cdots 0\;\underbrace{1\cdots 1}_{c} \ \ea \Bigg).
\ee
Note that for these three rows, we have $\alpha_i=\beta_{i-1}+1$,  $i\geq m-c+1$. So, we draw the conclusion for the first three rows of the partition.

This pairwise rows then continue. If the partition has an odd number of rows,  we draw the conclusion. If the partition has an even number of rows,  then the last row is even. By using formula(\ref{Lb2}), the contribution to symbol of the last row is
  $$\Bigg(\!\!\!\ba{c}0 \;\; 0\cdots 0\;\; 0 \cdots 0 \\
\;\;\;0\cdots 0\; \;1\cdots 1 \ \ea \Bigg),$$
which has no influence on the conclusion.

For a partition with only odd rows  in the $C_n$ theory, by using formula(\ref{oLc2}), the contribution to symbol of the first two odd rows $1^{2b+1}+1^{2c+1},(b>c),$ is
\be
 \Bigg(\!\!\!\ba{c}0\;\;0\cdots 0 \;\; \overbrace{1\cdots 1\;\;  1\cdots 1}^{b+1} \\
\;\;\;0\cdots 0 \;\; 0\cdots 0\;\underbrace{1\cdots 1}_{c} \ \ea \Bigg).
\ee
This pairwise rows then continue, which imply  the conclusion.

For a partition with only even rows  in the $C_n$ theory, by using formula(\ref{Lc2}), they contribute to symbol of the first two even row $1^{2b}+1^{2c},(b>c),$
\be
 \Bigg(\!\!\!\ba{c}0\;\;0\cdots 0 \;\; 0\cdots 0\;\; \overbrace{ 1\cdots 1}^{c} \\
\;\;\;0\cdots 0 \;\;\underbrace{ 1\cdots 1\;1\cdots 1}_{b} \ \ea \Bigg).
\ee
This pairwise rows then continue.  If the partition has an even number of row,  then we draw the conclusion. If the partition has an odd number of row,  the last row is even according to Proposition \ref{Pc}.
By using formula(\ref{Lc2}),  the contribution to symbol of the last row is
 $$\Bigg(\!\!\!\ba{c}0 \;\; 0\cdots 0\;\; 0 \cdots 0 \\
\;\;\;0\cdots 0\; \;1\cdots 1 \ \ea \Bigg)$$
which have no influence on  the conclusion.

For a partition in the $D_n$ theory,   the contribution to symbol of  the first row with $2m$ boxes is
  $$\Bigg(\!\!\!\ba{c}\overbrace{1 \;\; 1\cdots 1\;\; 1 \cdots 1}^m \\
\;\;\;0\cdots 0\; \;0\cdots 0 \ \ea \Bigg).$$
So, we draw the conclusion   for the first row.

According to  Proposition \ref{Pd}, the following two rows of the first row are either both of odd length or both of even. By using formula(\ref{Ld3}), the contributions to symbol of  two even rows $1^{2b}+1^{2c},(b>c),$ are
\be
 \Bigg(\!\!\!\ba{c}0\;\;0\cdots 0 \;\; 0\cdots 0\;\; \overbrace{ 1\cdots 1}^{c} \\
\;\;\;0\cdots 0 \;\; \underbrace{1\cdots 1\;1\cdots 1}_{b} \ \ea \Bigg).
\ee
 Note that for  these three rows, we have  $\alpha_i=\beta_{i+1}-1$ for $i\geq m-c+1$.  By using formula(\ref{oLd3}), the contributions to symbol of two odd rows $1^{2b+1}+1^{2c+1}, (b>c),$ is
\be
 \Bigg(\!\!\!\ba{c}0\;\;0\cdots 0 \;\; \overbrace{1\cdots 1\;\;  1\cdots 1}^{b+1} \\
\;\;\;0\cdots 0 \;\; 0\cdots 0\;\underbrace{1\cdots 1}_{c} \ \ea \Bigg).
\ee
So, we draw the conclusion for the first three rows in the partition.

This pairwise rows then continue. If the partition has an odd number of row,  we draw the conclusion. If the partition has an even number of row,  the last row is even. By using formula(\ref{Ld3}), the contribution to symbol of the last row has the following form
 $$\Bigg(\!\!\!\ba{c}0 \;\; 0\cdots 0\;\; 0 \cdots 0 \\
\;\;\;0\cdots 0\; \;1\cdots 1 \ \ea \Bigg)$$
which have no influence on  the conclusion.
\end{proof}

\begin{rmk}
The parity of  the first row  determine the structure of symbol.
\end{rmk}

Symbol can be seen as  an invariant of the partition.  This proposition imply  two  kinds of   maps preserving symbol possibly.  The first one is the map     between rigid partitions in the $B_n$ theory and rigid partitions  with only even rows     in the $C_n$ theory.  The second one is the  map     between rigid partitions  in the $D_n$ theory and rigid partitions with only odd rows    in the $C_n$ theory.    We propose one map for each case.

First, we propose the following map which take  a special rigid partition with only odd rows in the $B_n$ theory to  a special rigid partition (\ref{special}) with only even rows  in the $C_n$ theory
\bea \label{XS}
X_S:&& m^{2n_m+1}\, (m-1)^{2n_{m-1}}\, (m-2)^{2n_{m-2}  } \cdots 2^{2n_2} \, 1^{2n_1}  \non \\   & \mapsto&
m^{2n_m}\, (m-1)^{2n_{m-1} +2 }\, (m-2)^{2n_{m-2} - 2 } \cdots 2^{n_2+2} \, 1^{2n_1-2}\, ,
\eea
where $m$ has to be odd in order for the first object to be a partition in the $B_n$ theory.  It is clear that the map is a bijection so that $X_S^{-1}$ is well defined.  The map (\ref{XS}) is essentially the '$p_C$ collapse' described in \cite{CM93}. The inver map  $X_S^{-1}$ is essentially the '$p^B$ expansion'  described in \cite{CM93}.

Second, we propose the following map which take  a special rigid partition with only odd rows  in the $C_n$ theory to  a special rigid partition with only even rows  in the $D_n$ theory
\bea \label{YS}
Y_S:&& m^{2n_m+1}\, (m-1)^{2n_{m-1}}\, (m-2)^{2n_{m-2}  } \cdots 2^{2n_2} \, 1^{2n_1}  \non \\   & \mapsto&
m^{2n_m}\, (m-1)^{2n_{m-1} +2 }\, (m-2)^{2n_{m-2} - 2 } \cdots 2^{n_2-2} \, 1^{2n_1+2}\,.
\eea
 where $m$ has to be even in order for the first element to be a $C_k$ partition.   This is a bijection.   The map (\ref{XS}) is essentially the '$p_D$ collapse' described in \cite{CM93}. The inver map  $Y_S^{-1}$ is essentially the '$p^C$ expansion'  described in \cite{CM93}.

\subsection{$S$-duality maps for  rigid surface operators }
For the $B_n$ , $C_n$, and $D_n$ theories,  the  rigid semisimple conjugacy classes $S$  correspond to diagonal matrices with elements $+1$ and $-1$ along the diagonal \cite{GW08}. The   centraliser of  the  diagonal  matrices $S$ are as follows (at the Lie algebra level)
\bea
\so(2n{+}1) &\rar& \so(2k{+}1)\oplus \so(2n-2k) \,, \non \\
\spl(2n) &\rar& \spl(2k)\oplus \spl(2n-2k) \,, \\
\so(2n) &\rar& \so(2k)\oplus \so(2n-2k) \,. \non
\eea
The rigid semisimple surface operators  correspond to pairs of partitions  $(\la';\la'')$ \cite{GW08}.  In the $B_n$ case, $\la'$ is a rigid  $B_k$ partition and  $\la''$ is a  rigid $D_{n-k}$ partition. For the $C_n$ theories,  $\la'$ is a  rigid $C_k$ partition and  $\la''$ is a  rigid $C_{n-k}$ partition.  For the $D_n$ theories,   $\la'$ is a rigid $D_k$ partition and  $\la''$ is a rigid $D_{n-k}$ partition. In the theories under consideration, the rigid unipotent surface operators can be seen  as a special case  $\la''=0$.

The Langlands dual group of $B_n(SO(2n+1))$ is  $C_n(Sp(2n))$. Thus it is expected
$$BC:(\lambda_B;\rho_D)_B\rightarrow (\lambda'_C;\rho^{''}_C)_C$$
where the subscripts $B$ imply it is a partition in the $B_n$ theory. This map $BC$ preserve the symbol in the sense that the  addition of symbols on the two sides are equal
$$\sigma^B(\lambda)+\sigma^D(\rho)=\sigma^C(\lambda')+\sigma^C(\rho^{''}).$$
The group  $D_n(SO(n))$ is self-duality, which implies
$$DD:(\lambda_D;\rho_D)_D\rightarrow (\lambda'_D;\rho^{''}_D)_D.$$

Using maps $X_S$ and $Y_S$, we can propose dual maps of  rigid surface operators  between different theories.  The map $X_S$ imply that the special rigid unipotent surface operators in the $B_n$ and $C_n$ theories are related by $S$-duality. For  rigid unipotent surface operators in $B_n$ theory, the following map proposed in \cite{Wy09}
\begin{equation}\label{SSB}
WB:\,\,\,(\lambda_B, \emptyset)_B\rightarrow (\lambda_{odd}+\lambda_{even},\emptyset)\rightarrow (X_S\lambda_{odd},\lambda_{even})_C.
\end{equation}
First, we split the Young tableau into one Young tableau constructed from the odd rows  $\lambda_{odd}$ and one Young tableau constructed from the even rows  $\lambda_{even}$.   The first tableaux $\lambda_{odd}$  is always a special rigid partition in the $B_k$ theory and the map $X_S$ (\ref{XS}) turns this into a special rigid partition in the $C_k$ theory. While the second  partition $\lambda_{even}$  already corresponds to a special rigid partition in the $C_{n-k}$ theory and is left untouched.

For  rigid unipotent surface operators in $C_n$ theory, the following map is proposed
\begin{equation}\label{SSC}
 WC:\,\,\, (\lambda_C, \emptyset)_C \rightarrow (\lambda_{odd}+\lambda_{even},\emptyset)\rightarrow (X_S^{-1}\lambda_{even},Y_S\lambda_{odd})_B.
\end{equation}

For  rigid unipotent surface operators in $C_n$ theory, the following map is proposed
\begin{equation}\label{SSD}
  WD:\,\,\, (\lambda_D, \emptyset)_D \rightarrow (\lambda_{\mathrm{odd}}+\lambda_{\mathrm{even}},\emptyset)\rightarrow (\lambda_{\mathrm{even}},Y_S\lambda_{\mathrm{odd}})_D.
\end{equation}
These are more  maps  proposed in \cite{Wy09}\cite{ShO06}, which we refer the reader for more details.

Comparing the  table in Section \ref{tb}  with  that in Section \ref{tc},  we find that contributions to symbol of  the even rows in the $B_n$ theory  satisfy the same rules  with   even rows  in the $C_n$ theory.
So we have
\begin{equation}\label{zb}
\sigma^B(\lambda)=\sigma^B(\lambda_{odd})+\sigma^C(\lambda_{even}).
\end{equation}

Comparing  the table in Section \ref{tc} with   that in Section \ref{td},  we find that the odd rows  in the $C_n$ theory  satisfy same rules  with  odd rows  in the $D_n$ theory.
So we have
\begin{equation}\label{zd}
\sigma^D(\lambda)=\sigma^D(\lambda_{even})+\sigma^C(\lambda_{odd}).
\end{equation}

Using the maps (\ref{zb}), (\ref{zd}) and another two maps $X_S$, $Y_S$  preserving symbol, we can prove the following identities  proposed in \cite{ShO06}
\begin{itemize}
  \item For  a rigid semisimple surface operator  $(\lambda, \rho)$  in the  $B_n$ theory,  the symbol is
\begin{equation}\label{SBSB}
\sigma^B_{(\lambda,\rho)}=\sigma^B_{(\lambda)}+\sigma^D_{(\rho)}=\sigma^C_{(\lambda_{even})}+\sigma^C_{(X_S\lambda_{odd})}+\sigma^D_{(\rho_{even})}+\sigma^D_{(Y_S\rho_{odd})}.
\end{equation}

  \item   For  a rigid semisimple surface operator  $(\lambda, \rho)$  in the  $C_n$ theory,  the symbol is
\begin{equation}\label{SC}
\sigma^C_{(\lambda,\rho)}=\sigma^C_{(\lambda)}+\sigma^C_{(\rho)}=\sigma^C_{(\lambda_{even})}+\sigma^D_{(Y_s\lambda_{odd})}+\sigma^C_{(\rho_{even})}+\sigma^D_{(Y_S\rho_{odd})}.
\end{equation}

  \item
 For  a rigid semisimple surface operator  $(\lambda, \rho)$  in the  $D_n$ theory,  the symbol is
\begin{equation}\label{SD}
\sigma^D_{(\lambda,\rho)}=\sigma^D_{(\lambda)}+\sigma^D_{(\rho)}=\sigma^D_{(\lambda_{even})}+\sigma^D_{(Y_s\lambda_{odd})}+\sigma^D_{(\rho_{even})}+\sigma^D_{(Y_S\rho_{odd})}.
\end{equation}
\end{itemize}
Using the identities  (\ref{SBSB}), (\ref{SC}), and (\ref{SD}), we can prove the maps $WB$ (\ref{SSB}), $WC$(\ref{SSD}), and $WC$(\ref{SSD}) preserve symbol.

\section*{Acknowledgments}
We would like to thank  Zhisheng Liu and Ming Huang for  many helpful discussions.
This work was supported by a grant from  the Postdoctoral Foundation of Zhejiang Province.

%%%%%%%%%%%%%%%%%%%%%%%%%%%%%%%%%%%%%%%%%%


\begin{thebibliography}{9}

\bibitem{CM93}
D.~H. Collingwood and W.~M. McGovern, \emph{ Nilpotent orbits in semisimple Lie
  algebras},
\newblock Van Nostrand Reinhold, 1993.

\bibitem{GW06}
S.~Gukov and E.~Witten, \emph{Gauge theory, ramification, and the geometric
  {Langlands} program},  arXiv:hep-th/0612073

\bibitem{Wit07}
E.~Witten, \emph{Surface operators in gauge theory},   {\em Fortsch. Phys.}, {\textbf
  55} (2007) 545--550.
%%CITATION = FPYKA,55,545;%%.

\bibitem{GW08}
S.~Gukov and E.~Witten, \emph{Rigid surface operators},   arXiv:0804.1561

%%CITATION = HEP-TH/0612073;%%.

\bibitem{Wy09}
N.Wyllard, \emph{ Rigid surface operators and $S$-duality: some proposals},   arXiv: 0901.1833

\bibitem{ShO06}
B.~Shou, \emph{ Symbol, Rigid surface operaors and $S$-duality},  preprint, 26pp, arXiv: nnnmmm


\bibitem{Lu79}
G.~Lusztig, \emph{ A class of irreducible representations of a Weyl group},  Indag.Math, 41(1979), 323-335.

\bibitem{Lu84}
G.~Lusztig, \emph{ Characters of reductive groups over a finite field},
\newblock Princeton, 1984.


\bibitem{Sp92}
N.~Spaltenstein, \emph{Order relations on conjugacy classes and the
  {Kazhdan-Lusztig} map},  {\em Math. Ann.}, {\textbf 292} (1992) 281.

\bibitem{GNO76}
P.~Goddard, J.~Nuyts, and D.~I. Olive, \emph{Gauge theories and magnetic charge},
  {\em Nucl. Phys.}, {\textbf B125} (1977)
1.
%%CITATION = NUPHA,B125,1;%%.

\bibitem{AKS06}
P.~C. Argyres, A.~Kapustin, and N.~Seiberg, \emph{On {$S$-duality} for
  non-simply-laced gauge groups},  {\em JHEP}, {\textbf 06} (2006) 043,arXiv:hep-th/0603048

%%CITATION = HEP-TH 0603048;%%.

\bibitem{GM07}
  J.~Gomis and S.~Matsuura,
  \emph{Bubbling surface operators and $S$-duality},  \\
  {\em JHEP}, {\textbf 06} (2007) 025,arXiv:0704.1657

\bibitem{DGM08}
N.~Drukker, J.~Gomis, and S.~Matsuura, \emph{Probing $\cN=4$ SYM with surface
  operators},  {\em JHEP}, {\textbf 10} (2008) 048,  arXiv:0805.4199

\bibitem{GW14}
S.~Gukov, \emph{Surfaces Operators},   arXiv:1412.7145

\bibitem{Sh06}
B.~Shou,  \emph{Solutions of Kapustin-Witten equations for ADE-type groups}, preprint, 26pp, arXiv:1604.07172


\bibitem{SW17}
B.~Shou, and Q.~Wu, \emph{ Construction of the  Symbol Invariant of Partition}, preprint, 31pp,  arXiv:1708.07090


\bibitem{HW07a}
M.~Henningson and N.~Wyllard, \emph{Low-energy spectrum of {$\cN = 4$}
  super-{Yang-Mills} on {$T^3$}: flat connections, bound states at threshold,
  and {$S$-duality}},  {\em JHEP}, {\textbf 06} (2007), arXiv:hep-th/0703172

\bibitem{HW07b}
M.~Henningson and N.~Wyllard, \emph{Bound states in {$\cN = 4$} {SYM} on {$T^3$}:
  {$\Spin(2n)$} and the exceptional groups},  {\em JHEP}, {\textbf 07} (2007) 084, arXiv:0706.2803


\bibitem{HW08}
M.~Henningson and N.~Wyllard, \emph{Zero-energy states of $\cN = 4$ {SYM} on
  $T^3$: $S$-duality and the mapping class group},  {\em JHEP}, {\textbf 04} (2008)
  066, arXiv:0802.0660



\bibitem{Sh11}
B.~Shou, J.F.~Wu and M.~Yu, \emph{ AGT conjecture and AFLT states: a complete construction}, preprint, 28 pp., arXiv:1107.4784

\end{thebibliography}
\end{document}